\documentclass[a4paper,twoside,12pt,english]{article}

\usepackage[applemac]{inputenc}
\usepackage[english]{babel}
\makeatletter
\renewcommand\section{\@startsection
{section}{1}{0mm}%
{-2\bigskipamount}%
{\bigskipamount}%
{\normalfont\normalsize\bfseries}%
}

\newcommand\dateymd{\number\year, \ifcase\month\or
January\or February\or March\or April\or May\or June\or
July\or August\or September\or October\or November\or
December\fi, \number\day}
\newcommand\printtime{%
\c@hours=\time \divide\c@hours by60
\c@minutes=\c@hours \multiply\c@minutes by-60
\advance \c@minutes by \time
\ifnum\c@hours<10 0\fi\the\c@hours:%
\ifnum\c@minutes<10 0\fi\the\c@minutes}
\makeatother

\usepackage[OT1]{fontenc}

\usepackage{fancyhdr}
\usepackage{relsize}
\usepackage{bm} 
\usepackage{amssymb}
\usepackage{amsmath}
\usepackage{amsthm}
\usepackage{enumerate}
\usepackage{colortbl}
\usepackage{hyperref}

\usepackage{verbatim} 

\newcommand\skl[1]{{\,\sf#1\,}}
\newcommand\labelcell[1]{\cellcolor[gray]{0.8}\makebox[.9em][c]{\skl{#1}}}
\newcommand\hlinestrut{\hline\rule{0pt}{2.5ex}}

\newcommand\ie{i.\,e.~}
\newcommand\ifoi{\,\hbox{if\kern2.5pt and\kern2.5pt only\kern2.5pt if}\,{} }

\newcommand\df{\bfseries}
\newcommand\dfc[1]{\,{\df#1}\,}
\newcommand\dfd[1]{\,{\df#1}\hskip1pt}
\newcommand\secpar[1]{\S\,{#1}}

\newcommand\ensep{\unskip\hskip.65em\ignorespaces}

\newcommand\atilde{\lower3.5pt\hbox{\~{}}}
\newcommand\underl{\lower3.5pt\hbox{-}}

\newcommand\remarks{\bigskip\noindent\textit{Remarks}\par}
\newcommand\remark{\medskip\noindent\textit{Remark}.\hskip.5em}

\newcommand\pq[2]{\raise.25ex\hbox{\footnotesize${#1}\over{#2}$}%
\hskip-.35ex\null}
\newcommand\onehalf{\leavevmode\raise.5ex\hbox{\scriptsize$1$}\hskip-.33ex/\hskip-.3ex\lower.2ex\hbox{\scriptsize$2$}}
\newcommand\pqbis[2]{\leavevmode\raise.5ex\hbox{\scriptsize$#1$}\hskip-.33ex/\hskip-.3ex\lower.2ex\hbox{\scriptsize$#2$}}

\newcommand\halfsmallskip{\vskip0.5\smallskipamount}

\newcommand\xxxx[1]{%
 \hangindent2.5\parindent
 \hangafter1
 \noindent\hskip.5\parindent
 \hbox to2\parindent{\hss#1\hss}}

\newcommand\iim[1]{\xxxx{\textup{(#1)}}\ignorespaces}

\newcommand\ddd[1]{\halfsmallskip\vskip-2pt\noindent\hbox to 2\parindent{\hss\footnotesize$\bullet$\ \ }{#1}\ensep}

\newcommand\brwrap[1]{[\textsl{#1}\kern1pt]}

\makeatletter
\newcommand\bibref[1]{\@nameuse{b@#1}}
\renewcommand\@biblabel[1]{\brwrap{#1}}
\renewcommand\@cite[2]{\hbox{\brwrap{#1\if@tempswa\/\upshape\,:\,{\relscale{0.95}#2}\fi}}} 
\makeatother

\newcommand*\dbibref[2]{\bibref{#1}\,\textup{:\,{\relscale{0.95}#2}}}

\newcommand*\refco{\/\kern.1ex\textup{,}\hskip.45ex}
\newcommand*\refsc{\/\kern.15ex\textup{;} }

\newcommand\latop[2]{{#1\atop#2}}

\newcommand\sbset{\subset}

\newcommand\sbseteq{\subseteq}
\newcommand\spseteq{\supseteq}

\newcommand\cd[1]{\!#1\!}

\newcommand\isc{v^\ast}

\newcommand\rxi{\mathrel{\smash{\succ\kern-1.7ex\raise1.15ex\hbox{\mathsurround0pt$\scriptscriptstyle\xi$}\kern.4ex}}}
\newcommand\rxieq{\mathrel{\smash{%
 \vbox{\offinterlineskip\halign{\hfil##\hfil\cr
 \mathsurround0pt$\succ$\cr
 \noalign{\vskip-.5ex}%
 \mathsurround0pt$-$\cr
 \noalign{\vskip-1.15ex}%
 }\vss}\kern-1.05ex\raise1.15ex\hbox{\mathsurround0pt$\scriptscriptstyle\xi$}\kern.4ex}}}

\newcommand\psc{v^\pi}
\newcommand\pto{t^\pi}
\newcommand\scf{f} 

\newcommand\bbr{\mathbb{R}}

\newcommand\flr{\varphi}
\newcommand\phiset{Q}
\newcommand\phisetb{\hbox to1.75ex{\hss\hskip2pt$\overline{\hbox to1.5ex{\hss$Q$\hskip2pt\hss}}$\hss}}
\newcommand\lub{\hbox to1.75ex{\hss\hskip2pt$\overline{\hbox to1.5ex{\hss$F$\hskip2pt\hss}}$\hss}}
\newcommand\flrn{\flr^n}

\newcommand\ist{A}
\newcommand\xst{X}
\newcommand\yst{Y}
\newcommand\zst{Z}
\newcommand\wst{W}

\newcommand\rhobis{\widetilde\rho}

\newcommand\cst{C}

\newcommand\arank[1]{\bar r_{#1}}
\newcommand\rlr[1]{R_{#1}}

\newcommand\vbis{\widetilde v}
\newcommand\tbis{\widetilde t}
\newcommand\pscbis{\widetilde v{}^{\kern.75pt\pi}}

\newcommand\xibis{\smash{\widetilde{\hbox{\rule{0pt}{1.5ex}\smash{$\xi$}}}}}

\newcommand\tref[1]{\smash{\hbox{$\widetilde{\hbox{\ref{#1}}}$}}}

\newcommand\vaa{\mathsf{V}}
\newcommand\vaasub{\vaa\kern-1pt}
\newcommand\vxx{\vaasub_{\scriptscriptstyle X\kern-1pt X}}
\newcommand\vxy{\vaasub_{\scriptscriptstyle X\kern-.25pt Y}}
\newcommand\vyx{\vaasub_{\scriptscriptstyle Y\kern-1pt X}}
\newcommand\vyy{\vaasub_{\scriptscriptstyle Y\kern-.25pt Y}}
\newcommand\vrs{\vaasub_{\scriptscriptstyle R\kern-.25pt S}}
\newcommand\zeromatrix{\mathsf{O}}
\newcommand\irreopen{{\cal I}}
\newcommand\flrx{\flr_{\scriptscriptstyle X}}
\newcommand\flrax{\flr_{\scriptscriptstyle A\setminus X}}
\newcommand\flry{\flr_{\scriptscriptstyle Y}}
\newcommand\flrr{\flr_{\scriptscriptstyle R}}
\newcommand\gflr{\psi}
\newcommand\gflrx{\psi_{\scriptscriptstyle X}}
\newcommand\gflry{\psi_{\scriptscriptstyle Y}}
\newcommand\fxx{F_{\scriptscriptstyle X\kern-1pt X}}
\newcommand\fyy{F_{\scriptscriptstyle Y\kern-.25pt Y}}
\newcommand\fxy{F_{\scriptscriptstyle X\kern-.25pt Y}}
\newcommand\fzz{F_{\scriptscriptstyle Z\kern-1pt Z}}
\newcommand\frr{F_{\scriptscriptstyle R\kern-.25pt R}}
\newcommand\flrbis{\hbox to1.75ex{\hss\hskip1.25pt$\widetilde{\hbox to1.5ex{\hss$\varphi$\hskip1.25pt\hss}}$\hss}}
\newcommand\vaabis{\smash{\widetilde{\hbox{\vphantom{t}\smash{$\vaa$}}}}}
\newcommand\vaabissub{\vaabis\kern-1pt}
\newcommand\vaabisxx{\vaabissub_{\scriptscriptstyle X\kern-1pt X}}
\newcommand\waa{\mathsf{W}}
\newcommand\waabis{\smash{\widetilde{\hbox{\vphantom{t}\smash{$\waa$}}}}}
\newcommand\xstbis{\hbox to1.97ex{\hss\hskip2.5pt$\smash{\widetilde{%
 \hbox to1.9ex{\hss\vphantom{t}\smash{$\xst$}\hskip2.5pt\hss}}}$\hss}}
\newcommand\xstbiss{\smash{\widetilde\xst}}
\newcommand\flrbissub{\flrbis\kern-1pt}
\newcommand\flrxbis{\flrbissub_{\scriptscriptstyle\xstbiss}}
\newcommand\flrbisX{\flrbissub_{\scriptscriptstyle\xst}}

\newcommand\flrbisx{\flrbissub_x}
\newcommand\flrbisy{\flrbissub_y}
\newcommand\flrbisz{\flrbissub_z}
\newcommand\flraxbis{\flrbissub_{\scriptscriptstyle A\setminus\xstbiss}}

\newcommand\xsth{\smash{\widehat X}}

\newcommand\xh{\hat x}
\newcommand\yh{\hat y}

\newcommand\xstp{X'}
\newcommand\ystp{Y'}

\newcommand\fbis{\hbox to1.97ex{\hss\hskip2.5pt$\smash{\widetilde{%
 \hbox to1.9ex{\hss\vphantom{t}\smash{$F$}\hskip2.5pt\hss}}}$\hss}}
\newcommand\fbisxx{\fbis_{\kern-2pt\scriptscriptstyle \xstbiss\kern-1pt\xstbiss}}

\newcommand\chrel{\mathrel{\trianglerighteq}}
\newcommand\chrels{\mathrel{\equiv}}
\newcommand\chrela{\mathrel{\hbox{\relscale{1.25}$\kern1pt\triangleright\kern1pt$}}}
\newcommand\chrelbis{\mathrel{\smash{\widetilde{\hbox{\vrule width0pt height6.5pt\smash{$\trianglerighteq$}}}}}}
\newcommand\chrelabis{\mathrel{\smash{\widetilde{\hbox{\vrule width0pt height6.5pt\smash{\hbox{\relscale{1.25}$\kern1pt\triangleright\kern1pt$}}}}}}}

\newcommand\fun{G}
\newcommand\funbis{\widehat G}
\newcommand\funter{H}
\newcommand\funiv{S}

\newtheorem{proposition}{Proposition}[section]
\newtheorem{lemma}[proposition]{Lemma}
\newtheorem{theorem}[proposition]{Theorem}
\newtheorem{corollary}[proposition]{Corollary}
\newtheorem{open}{Open question}

\newlength\repskip 
{\null\hskip-\repskip\hbox to0.9\hsize\bgroup\hbox to0pt{\small$(\ref{#1})$\hss}\hfil\hskip.1\hsize$\displaystyle}%
{$\hfil\egroup}

{\hbox to\hsize\bgroup\hbox to0pt{\small$(\ref{#1})$\hss}\hfil$\displaystyle}%
{$\hfil\egroup\nonumber}

\newcommand\cand[1]{\textsf{#1}}
\newcommand\csep{,\kern1pt}

\begin{document}

\thispagestyle{empty}

\null\vskip-15mm\null 

\bgroup
\noindent\hskip-0.05\textwidth\vbox{\hsize=1.1\textwidth%
\begin{center}
\hrule
\vskip7.5mm
\textbf{\uppercase{Fraction-like ratings from preferential voting}}
\par\medskip
\textsc{Rosa Camps,\, Xavier Mora \textup{and} Laia Saumell}
\par
Departament de Matem\`{a}tiques,
Universitat Aut\`onoma de Barcelona,
Catalonia
\par\medskip
\texttt{xmora\,@\,mat.uab.cat}
\par\medskip
Revised 26th March 2014
\vskip7.5mm
\hrule
\end{center}
\egroup

\vskip-6mm\null
\begin{abstract}
A method is given for resolving a matrix of preference scores into a well-specified mixture of options. This is done in agreement with several desirable properties, including the continuity of the mixing proportions with respect to the preference scores and a condition of compatibility with the Condorcet-Smith majority principle. These properties are achieved by combining the classical rating method of Zermelo with a projection procedure introduced in previous papers of the same authors. 


\vskip2pt
\bigskip\noindent
\textbf{AMS subject classifications:} 
\textit{%
05C20, 
91B12, 
91B14, 
91C15, 
91C20. 
}
\end{abstract}

\vskip7.75mm
\hrule

\section*{}
A vote is an expression of the preferences of several individuals about certain options with a view towards reaching a common decision.
Generally speaking, the decision need not be choosing a single option, but 
it~can also take the form of mixing a number of them according to certain proportions.
For instance, one could be dividing a prize among several contenders,
or a budget among several items.
This article is aimed at a method
for suitably \emph{determining the proportions of such mixed collective choices.}

The input from which we set ourselves to derive these proportions or mixing fractions is the \emph{matrix of preference scores} of Ramon Llull and Condorcet~\cite[\secpar{3},\,\secpar{7}]{mu},
\ie the matrix that compares each option to every other in~terms of the number of voters
who prefer the former to the latter.

Assume, for instance, that a committee of 18 people must decide how to distribute a budget among
four items \cand{a\csep b\csep c\csep d} and that they express the following preferences:
\begin{equation}
\label{eq:ex1}
10:\cand{a}\!>\!\cand{b}\!>\!\cand{c}\!>\!\cand{d},\quad
3:\cand{b}\!>\!\cand{c}\!>\!\cand{d}\!>\!\cand{a},\quad
3:\cand{c}\!>\!\cand{d}\!>\!\cand{b}\!>\!\cand{a},\quad
2:\cand{d}\!>\!\cand{b}\!>\!\cand{a}\!>\!\cand{c}.
\end{equation}
The number in front of each ranking indicates  how many people expressed~it.
One can work out that \cand{a} is preferred to \cand{b} by 10 people against 8,
\cand{b} is preferred to \cand{c} by 15 against 3, et cetera.
These numbers are collected in the following table, that we call the Llull matrix of the vote:%
\footnote{Since we are interested only in the preferences of $x$ over $y$ for $x\neq y$, we use the diagonal cells for specifying the simultaneous labelling of rows and columns by the existing options. The cell located in row $x$ and column $y$ gives information about the preference of $x$ over $y$.}
\smallskip
\begin{equation}
\label{eq:llull1}
\begin{small}
\begin{tabular}{|c|c|c|c|}
\hlinestrut
\labelcell{a}&10&12&10\\
\hlinestrut
8&\labelcell{b}&15&13\\
\hlinestrut
6&3&\labelcell{c}&16\\
\hlinestrut
8&5&2&\labelcell{d}\\
\hline
\end{tabular}
\end{small}
\,.
\end{equation}
In which proportions should the budget be divided?

\medskip
Notice that the individual votes cannot be recovered from the Llull matrix. Therefore, our setting is not suitable for the purpose of proportional representation, which has to do with mapping the electorate onto the elected options and therefore requires more information than just the Llull matrix of the vote. However, our setting still seems appropriate for distributing a prize or a budget between different options in accordance with their relative merits as summarized in the Llull matrix.

\medskip
In the preceding example the preferential information is complete: since every voter has ordered all the options, the preference scores for any ordered pair of options and its opposite add up to the total number of voters. Generally speaking, however, it need not be so. For instance, voters could give only truncated rankings, where no preferences are expressed between the non-mentioned options. 
The method that we are looking for should be able to deal also with 
such situations of \emph{incomplete preferences}. An extreme case is that where every voter confines to choosing a single option. In this case, the mixing fractions should certainly coincide with the respective vote fractions. We will refer to this requirement as \dfd{single-choice voting consistency}.

Assume, for instance, that 100~voters express themselves in the following way:
\begin{equation}
\label{eq:ex2}
54:\cand{a},\quad
22:\cand{b},\quad
13:\cand{c},\quad
11:\cand{d}.
\end{equation}
That is, 54~voters express their preference for \cand{a} over any of the three other options, but they do not give any information about their preferences between \cand{b\csep c} and \cand{d}; the other voters act similarly in connection with other options. One easily checks that the Llull matrix of this vote takes the following form:
\smallskip
\begin{equation}
\label{eq:llull2}
\begin{small}
\begin{tabular}{|c|c|c|c|}
\hlinestrut
\labelcell{a}&54&54&54\\
\hlinestrut
22&\labelcell{b}&22&22\\
\hlinestrut
13&13&\labelcell{c}&13\\
\hlinestrut
11&11&11&\labelcell{d}\\
\hline
\end{tabular}
\end{small}
\,.
\end{equation}
The condition of single-choice voting consistency requires that in such a situation the mixing fractions shoud be $(0.54,\, 0.22,\, 0.13,\, 0.11).$

%

\medskip
Priority ratings are often used only for ranking purposes.
However, in this article we are interested in mixing fractions per se,
\ie as an expression of which specific share of prize or burden should be allotted to every option.
In consonance with such a \emph{quantitative character,} we require
a \dfc{continuous dependence} of the mixing fractions on the preference scores.


On the other hand, the mixing-fraction character that we are looking for calls also for
the following condition of \dfd{unanimous decomposition},
that we divide in two parts:
\label{txt:ud}
(a)~If every option from a set $\xst$ is unanimously preferred to any option from outside~$\xst$,
then the mixing fractions should vanish outside of $\xst$.
In particular,  if an option is unanimously preferred to any other, then it should get a mixing fraction equal to 1 and all the other mixing fractions should be equal to 0. 
(b)~In the complete case the following converse statement should hold too:
If $\xst$ is the set of options that get non-vanishing fractions,
then each option from $\xst$ is unanimously preferred to any option from outside $\xst;$
besides, $\xst$ is the minimal set with this property.

Consider, for instance, the preferences
\begin{equation}
\label{eq:ex3}
60:\cand{a}\!>\!\cand{b}\!>\!\cand{c}\!>\!\cand{d},\quad
40:\cand{b}\!>\!\cand{a}\!>\!\cand{d}\!>\!\cand{c},\quad
\end{equation}
whose corresponding Llull matrix is
\smallskip
\begin{equation}
\label{eq:llull3}
\begin{small}
\begin{tabular}{|c|c|c|c|}
\hlinestrut
\labelcell{a}&60&100&100\\
\hlinestrut
40&\labelcell{b}&100&100\\
\hlinestrut
0&0&\labelcell{c}&60\\
\hlinestrut
0&0&40&\labelcell{d}\\
\hline
\end{tabular}
\end{small}
\,.
\end{equation}
According to the condition of unanimous decomposition,
in~such a situation the only options with non-vanishing fractions should be \cand{a} and \cand{b}.
Notice also that in the case of example (\ref{eq:ex1}--\ref{eq:llull1}) part~(b) of the condition of unanimous decomposition requires that every option should receive a positive fraction.

\medskip
Our problem can be seen as a special case of a more general one where a matrix of paired-comparison scores is to be summed up into a~set of priority ratings
(not necessarily with the character of mixing fractions).
\ensep
Such a problem arises not only in preferential voting, but also in sport tournaments, psychometrics,
multi-criteria decision theory, web search engine rankings, et cetera.
See for instance \cite{che:1998,gonz,lm:2012}.




However, voting has a special character in that the comparisons
are decided by human individuals.
Because of this, it becomes advisable to comply with certain 
\emph{majority principles}. In the paired-comparison setting, the standard formulation is the \emph{Condorcet principle} 
(see \cite[ch.\,1, \secpar4.2]{mu}, \cite[\secpar7.2]{nitzan} and \cite[p.\,153--154]{t6}):
If the preference scores of a~particular option over the others are all of them \emph{greater than
half} the number of voters, then that option should be socially preferred to any other.
In our setting, being socially preferred means simply getting a larger fraction.
\ensep
More generally, we will consider also the following extended version,
which was introduced in 1973 by John H.~Smith \cite[\secpar 5]{smith}
(except for the provision of vanishing fractions)
and will be referred to as \dfd{Condorcet-Smith principle}: 
If~the options are partitioned in two sets $\xst$ and $\yst$ so that
every member of $\xst$ is preferred to any member of $\yst$ by \emph{more than half} of the voters,
then every member of $\xst$ should get a larger fraction than any member of $\yst$
unless both fractions vanish.

According to this condition, in the case of example (\ref{eq:ex1}--\ref{eq:llull1}) the mixing fractions should decrease along the order $\cand{a}\!>\!\cand{b}\!>\!\cand{c}\!>\!\cand{d}$
(take successively $\xst=\{\cand{a}\},\,\{\cand{a},\cand{b}\},\,\{\cand{a},\cand{b},\cand{c}\}$).

\medskip
As we will see,
the conditions of single-choice voting consistency, continuity and unanimous decomposition are satisfied by a celebrated method that was introduced in 1929 by Ernst Zermelo
in the context of chess tournaments~\cite{ze}.
%
However, Zermelo's method by itself does not comply with the
Condorcet principle.
In~fact, it need not
give the largest fraction to an option that is placed first by more than half of the voters.
For instance, in the case of example (\ref{eq:ex1}--\ref{eq:llull1}) it gives the following fractions: 
\cand{a}: 0.303, \cand{b}: 0.387, \cand{c}: 0.201 \cand{d}: 0.109, where \cand{b} gets the largest fraction in spite of the fact that \cand{a} has a majority of first placings.
As we will show in this article, this problem disappears when Zermelo's method is preceded by the 
\dfc{CLC~projection} that is introduced in~\cite{crc, cri}
(`CLC' stands for ``Continuous Llull-Condorcet'').
In~the case of (\ref{eq:ex1}--\ref{eq:llull1}), this combined procedure gives the following results: \cand{a}:~0.323, \cand{b}: 0.288, \cand{c}: 0.217 \cand{d}:~0.173.

\pagebreak 

The resulting method, that is, the CLC~projection followed by Zermelo's method,
combines, among others, the following properties:
fraction character, including the above-mentioned conditions of single-choice voting consistency and unanimous decomposition,
continuity with respect to the original preference scores, 
and compliance with the Condorcet-Smith principle.
\ensep
To~our knowledge, the previous literature does not offer any other rating method with these properties.

The reader interested to try the proposed method can use the \,\textsl{CLC~calculator}\, which has been made available at~\cite{clc-calculator}.

\medskip
This article is structured as follows: In Section~1 we introduce some general terminology and notation. Section~2 is devoted to Zermelo's method by itself, with some new results, especially in connection with the continuous dependence of the ratings on the data in the reducible case.
Section~3 looks at certain properties of the paired-comparison matrices that arise from the CLC projection of \cite{crc,cri}. Section~4 combines the previous results to show that the concatenation of the CLC projection and Zermelo's method achieves the desired properties. Finally, in Section~5 we ask ourselves for the possibility of other methods with the same properties and we discuss some related questions.

\section{Terminology and notation}

\paragraph{1.1}
We consider a finite set~$\ist$. Its elements represent the options which are the matter of a vote. 
The number of elements of $\ist$ will be denoted by~$N$.
We will be based upon the numbers of voters who expressed a preference for $x$ over $y$,
where $x$ and $y$ vary over all ordered pairs of different options.
These numbers will be denoted by $V_{xy}$. 
Instead of them, most of the time we will be dealing with the fractions $v_{xy}=V_{xy}/V$,
where $V$ denotes the total number of votes.
We will refer to $V_{xy}$ and $v_{xy}$ respectively
as the absolute and relative \dfc{preference scores} associated with the ordered pair $xy,$ 
and the whole collection of these scores will be called the (absolute or relative)
\dfc{Llull matrix} of the vote.

\smallskip
The preference scores are obviously bound to satisfy the inequality
\begin{equation}
\label{eq:sumle1}
v_{xy}+v_{yx}\,\le\,1.
\end{equation}
A~matrix of preference scores satisfying $v_{xy}+v_{yx}=1$ for any $x$ and $y$
will be said to be \dfd{complete}.


\medskip
Incomplete Llull matrices arise when preferences are not expressed by some voters on some pairs of options. In this connection, one must be careful to distinguish a definite indifference about two options from a lack of information about them.
One voter who expresses a definite indifference about $x$ and $y$ should be considered equivalent to half a voter preferring $x$ to $y$\, plus another half a voter preferring~$y$~to~$x$.
In contrast, a voter who gives no information about whether he prefers $x$ to $y$ or viceversa should be counted neither in $V_{xy}$ nor in $V_{yx}.$

In this spirit, a ballot that confines to choosing a single option should be interpreted as expressing \emph{nothing else} than a preference for that option over any other.
Therefore, in the case of single-choice voting
---where everybody votes in this way---
the Llull matrix takes the form $v_{xy} = \scf_x$ for any $y\neq x$, where $\scf_x$ is the fraction of the vote that chooses~$x$.

%
%
%
%
%
%

\medskip
Besides the scores $v_{xy}$, in the sequel we will often deal with the
\dfc{margins} $m_{xy}$ and the \dfc{turnouts}~$t_{xy}$,
which are defined respectively by
\begin{equation}
m_{xy} \,=\, \hbox to26mm{$v_{xy} - v_{yx},$\hfil}\qquad
t_{xy} \,=\, \hbox to26mm{$v_{xy} + v_{yx}.$\hfil}
\end{equation}
Obviously, their dependence on the pair $xy$ is respectively antisymmetric and symmetric, that is
\begin{equation}
m_{yx} \,=\, \hbox to26mm{$- m_{xy},$\hfil}\qquad
t_{yx} \,=\, \hbox to26mm{$t_{xy}.$\hfil}
\end{equation}
It is clear also that the scores $v_{xy}$ and $v_{yx}$
can be recovered from $m_{xy}$ and $t_{xy}$ by means of the formulas
\begin{equation}
v_{xy} \,=\, \hbox to26mm{$(t_{xy} + m_{xy})/2,$\hfil}\qquad
v_{yx} \,=\, \hbox to26mm{$(t_{xy} - m_{xy})/2.$\hfil}
\end{equation}

\smallskip
Instead of the margins $m_{xy} = v_{xy} - v_{yx}$, sometimes, especially in decision theory, one considers the \dfc{ratios} $p_{xy} = v_{xy}/v_{yx}$ (which requires the preference scores to be all of them positive).
\ensep
Alternatively, one can consider the \dfc{relative scores} $q_{xy} = v_{xy}/t_{xy}$ (which only requires the turnouts to be positive).
Obviously, the matrix of relative preference scores is always complete.
\ensep
The ratios and the relative scores are related to each other by the formulas $p_{xy} = q_{xy}/(1-q_{xy})$,
$q_{xy} = p_{xy}/(1+p_{xy})$.
\ensep
Notice however that in the incomplete case neither the margins, nor the ratios, nor the relative scores, allow to recover the original scores, unless one knows also the turnouts $t_{xy}$.


\smallskip
In order to refer to it as a whole, the Llull matrix made of the preference scores $v_{xy}$ will be denoted as $(v_{xy})$, or alternatively as $\vaa$.
\ensep
We will also
use the notation $\vrs$ to mean the restriction of $(v_{xy})$ to~$x\in R$ and $y\in S$, where $R$ and $S$ are arbitrary non-empty subsets of $\ist$.
\ensep
Similarly, if $(u_x)$ is a collection of numbers indexed by $x\in\ist$, its restriction to~$x\in R$ will be denoted as~$u_{\scriptscriptstyle R}$.

\paragraph{1.2}
The simplest
rating of the overall acceptance of an option~$x$ is its \dfd{mean preference score}, that is, the arithmetic mean of its preference scores against all the other options:
\begin{equation}
\label{eq:rrates}
\rho_x \,=\, {\frac{\hbox{\small1}}{\hbox{\small{$N-1$}}}}\,\sum_{y\neq x} v_{xy}.
\end{equation}
This quantity is linearly related to the rank-based count proposed in~1433 by Nikolaus von Kues \cite[\secpar{1.4.3}, \secpar{4}]{mu}
and again in~1770--1784 by Jean-Charles de Borda \cite[\secpar{1.5.2}, \secpar{5}]{mu} 
(both of them being restricted to the complete case).
More specifically,
their count amounts to $1+(N-1)\rho_x = (1-\rho_x) + \rho_x\,N$.
Instead of it, in \cite{crc,cri} we considered the \dfc{mean ranks}
$\arank{x}$, which are given by
\begin{equation}
\label{eq:arank}
\arank{x} \,=\, N-(N-1)\,\rho_x \,=\, \rho_x + (1-\rho_x)\,N.
\end{equation}
Notice that, contrarily to $\rho_x$, lower mean ranks correspond to a higher acceptance.
The ratings~$\rlr{x}$ that were considered in \cite{crc, cri}
are nothing else than the mean ranks that are obtained
after transforming the Llull matrix
by means of the CLC~projection.

The mean preference scores $\rho_x$ can certainly be rescaled to add up to $1$. More interestingly,
in the case of single-choice voting they fulfil the requirement
of coinciding with the vote fractions $\scf_x$.
In~fact, having $v_{xy} = \scf_x$ for any $y\neq x$ certainly implies $\rho_x=\scf_x$.
\ensep
However, they definitely do not satisfy the condition of unanimous decomposition. For instance, for $\ist=\{a,b,c\}$ with $v_{ab} = v_{ac} = v_{bc} = 1$ (which implies $v_{ba} = v_{ca} = v_{cb} = 0$)
$a$ is unanimously preferred to both $b$ and $c$
but $\rho_b = 1/2 > 0$.

\paragraph{1.3}
As we mentioned in the introduction, our goal will be achieved by means of Zermelo's method
together with a prior application of the CLC projection.
In connection with Zermelo's method,
we need to introduce a qualitative notion of priority
that also bears relation to the CLC projection.
In order to define it, 
we will make use of the \dfc{indirect scores} $\isc_{xy}$:
given $x$ and $y$, one considers all possible paths $x_0 x_1 \dots x_n$
going from~$x_0 = x$ to~$x_n = y$; every such path is associated with the
score of its weakest link, \ie the smallest value of $v_{x_ix_{i+1}}$;
finally, $\isc_{xy}$ is defined as the maximum value of this associated score
over all paths from $x$ to $y$.
In other words,
\begin{equation}
\isc_{xy} \hskip.75em = \hskip.75em
\max_{\vtop{\scriptsize\halign{\hfil#\hfil\cr\noalign{\vskip.5pt}$x_0=x$\cr$x_n=y$\cr}}}
\hskip.75em
\min_{\vtop{\scriptsize\halign{\hfil#\hfil\cr\noalign{\vskip-1.25pt}$i\ge0$\cr$i<n$\cr}}}
\hskip.75em v_{x_ix_{i+1}},
\label{eq:paths}
\end{equation}
where the \,$\max$\, operator considers all possible paths from $x$ to $y$,
and the \,$\min$\, operator considers all the links of a particular path.
For instance, the indirect scores for the Llull matrix (\ref{eq:llull1}) are as follows:
\begin{equation}
\label{eq:llull1-indirectscores}
(V^*_{xy}) \,=\, 
\begin{small}
\begin{tabular}{|c|c|c|c|}
\hlinestrut
\labelcell{a}&10&12&12\\
\hlinestrut
8&\labelcell{b}&15&15\\
\hlinestrut
8&8&\labelcell{c}&16\\
\hlinestrut
8&8&8&\labelcell{d}\\
\hline
\end{tabular}
\end{small}
\,.
\end{equation}

\medskip
By the definition of $\isc_{xy}$, the inequality $\isc_{xy}>0$ clearly defines a transitive relation.
In the following we will denote it by the symbol \,$\chrel$. Thus,
\begin{equation}
x\chrel y \ \Longleftrightarrow\  \isc_{xy}>0.
\end{equation}
Associated with it, it is interesting to consider also the following derived relations,
which keep the property of transitivity and are respectively symmetric and asymmetric:
\begin{align}
x\chrels y \ &\Longleftrightarrow\  \isc_{xy}>0 \text{ \,and\, } \isc_{yx}>0,\\[2.5pt]
x\chrela y \ &\Longleftrightarrow\  \isc_{xy}>0 \text{ \,and\, } \isc_{yx}=0.
\end{align}
Therefore, $\chrels$ is an equivalence relation
and $\chrela$ is a partial order.
In the following, the situation where $x\chrela y$ will be expressed by saying that $x$ \dfc{dominates} $y$.

The equivalence classes of~$\ist$ by $\chrels$ are called the \dfc{irreducible components} of~$\ist$ (for~$\vaa$). If there is only one of them, namely $\ist$ itself, then one says that the matrix $\vaa$ is irreducible. So, $\vaa$ is irreducible \ifoi $\isc_{xy}>0$ for any $x,y\in\ist$. It is not difficult to see that this property is equivalent to the following one formulated in terms of the direct scores only: there is no splitting of $\ist$
into two classes $\xst$ and~$\yst$ so that $v_{yx}=0$ for any $x\in\xst$ and $y\in\yst$; in other words, there is no ordering of $\ist$ for which the matrix~$\vaa$ takes the form
\begin{equation}\label{eq:blocs}
\begin{pmatrix} \vxx & \vxy \\ \zeromatrix & \vyy \end{pmatrix},
\end{equation}
where $\vxx$ and $\vyy$ are square matrices and $\zeromatrix$ is a zero matrix.
Besides, a subset $\xst\sbseteq\ist$ is an irreducible component \ifoi $\xst$ is maximal, in~the sense of set inclusion, for the property of $\vxx$ being irreducible.
\ensep
On~the other hand,
it~also happens that the relation \,$\chrela$\, is compatible with the equivalence relation \,$\chrels$, \ie if $x\chrels\bar x$ and $y\chrels\bar y$ then \,$x\chrela y$ implies $\bar x\chrela\bar y$.
As a consequence, the relation \,$\chrela$\, can be applied also to the irreducible components of $\ist$ for $\vaa$.
In~the following we will be interested in the case where $\vaa$ is irreducible, or more generally, when there is a \dfd{top dominant irreducible component}, \ie an irreducible component which dominates any other irreducible component.


\section{Zermelo's method of strengths}
\label{sec:zms}

The Llull matrix of a vote among $V$ voters can be viewed as
a tournament between the members of~$\ist$. In fact, it is as if $x$ and $y$ had played $T_{xy}=t_{xy}V$ matches (the number of voters who made a comparison between $x$ and $y$, even if some of these voters considered $x$ at the same level as $y$) and $V_{xy}=v_{xy}V$ of these matches had been won by $x$, whereas the other $V_{yx}=v_{yx}V$ had been won by~$y$ (one tied match is counted as half a match in favour of $x$ plus half a match in favour of $y$).
It was in such a scenario that Zermelo devised in 1929 his rating method~\cite{ze}.
Later on, the same method has been rediscovered by several other autors (see~\cite{st,ke} and the references therein).

\medskip
Zermelo's method is based upon a probabilistic model for the outcome of a match between two items $x$ and $y$. This model assumes that such a match is won by $x$ with probability $\flr_x/(\flr_x+\flr_y)$ whereas it is won by $y$ with probability $\flr_y/(\flr_x+\flr_y)$, where $\flr_x$ is a non-negative parameter associated with each player~$x$, usually referred to as its {\df strength}. If~all matches are independent events, the probability of obtaining a particular system of values for the scores $(V_{xy})$ 
is~given by
\begin{equation}
P \,=\, \prod_{\{x,y\}}\,\left(\vbox{\halign{\hfil#\hfil\cr$T_{xy}$\cr$V_{xy}$\cr\noalign{\vskip-11pt}}}\right)
\left(\frac{\flr_x}{\flr_x+\flr_y}\right)^{\hskip-3ptV_{xy}}
\left(\frac{\flr_y}{\flr_x+\flr_y}\right)^{\hskip-3ptV_{yx}},
\label{eq:probabilitat}
\end{equation}
where the product runs through all unordered pairs $\{x,y\}\sbseteq\ist$ with $x\neq y$. Notice that $P$~depends only on the strength ratios; in other words, multiplying all the strengths by the same value has no effect on the result. On account of this, we will normalize the strengths by requiring their sum to be equal to~$1$.
\ensep
In order to include certain extreme cases, one must allow for some of the strengths to vanish. However, this may conflict with $P$ being well defined, since it could lead to indeterminacies of the type $0/0$ or $0^0$. 
So,~one should be careful in connection with vanishing strengths.
\ensep
With all this in mind, for the moment we will let the strengths vary in the following set:
\begin{equation}
\phiset = \{\,\flr\in\bbr^\ist\mid\flr_x>0\text{ \,for all }x\in\ist,\ \sum_{x\in\ist}\flr_x = 1\,\}.
\label{eq:conjuntphis}
\end{equation}
Together with this set, in the following we will consider also its closure~$\phisetb$, which includes vanishing strengths, and its boundary $\partial\phiset = \phisetb\setminus\phiset$.
\ensep
As it will be seen below, 
Zermelo's method corresponds to
a maximum likelihood estimate of the parameters~$\flr_{x}$ from a given set of actual values of~$V_{xy}$ (and of $T_{xy}=V_{xy}+V_{yx}$).
In other words, given the values of~$V_{xy}$, one looks for the values of $\flr_x$ which maximize the probability~$P$.

\medskip
The hypothesis of independence which lies behind formula~(\ref{eq:probabilitat}) is certainly not satisfied by the binary comparisons which arise out of preferential voting.
However, it turns out that
the same estimates of the parameters~$\flr_{x}$ arise 
from a related model where
the voters are assumed to express complete definite rankings
(‘definite’ means here ‘without ties’).
\ensep
Both Zermelo's binary model and the ranking model that we are about to introduce
can be viewed as special cases of a more general model,
proposed in 1959 by~Robert Duncan Luce, which considers the outcome of making a choice out~of~multiple options \cite{luce}.
According to Luce's ``choice axiom'', the~probabilities of two different choices $x$ and $y$ \,are in a ratio which does not depend on~which other options are present.
As a consequence, it follows that every option $x$ can be associated a~number~$\flr_x$ so that the probability of~choosing $x$ out of a set~$\xst$ that contains $x$ is given by $\flr_x/(\sum_{y\in\xst}\flr_y)$.
\ensep
Obviously, Zermelo's model corresponds to considering binary choices only.
\ensep
However, Luce's model also allows
to associate every complete definite ranking with a certain probability. In~fact, such a ranking can be viewed as the result of first choosing the winner out of the whole set~$\ist$, then choosing the best of the remainder, and so on. If~these successive choices are assumed to be independent events, then one can easily figure out the corresponding probability.
\ensep
Furthermore, one can see that these probabilities make the expected rank of~$x$
equal to~$E(r_x) = N - \sum_{y\ne x}\flr_x/(\flr_x+\flr_y)$.\linebreak 
By~equating these values to the mean ranks given by equations (\ref{eq:rrates}--\ref{eq:arank}), namely $\arank{x} = N - \sum_{y\neq x} v_{xy}$
---so using the so-called method of moments---
one obtains exactly the same equations for the estimated values of the parameters~$\flr_{x}$ as in the method of maximum likelihood, namely equations (\ref{eq:fratesZ}) below.
\ensep
Notice also that, in accordance with Luce's theory of choice,
the normalization condition $\sum_{x\in\ist}\flr_x = 1$ 
allows to view $\flr_x$ as the first-choice probability of~$x$
(among non-abstainers).
\ensep
Anyway, \ie independently of the reasons behind them, the resulting values of $\flr_x$
will be seen to have good properties for our purposes.

\medskip
In the following we take the point of view of maximum likelihood. So, given the values of~$V_{xy}$, we will look for the values of $\flr_x$ which maximize the probability~$P$. Since $V_{xy}$ and $T_{xy}=V_{xy}+V_{yx}$ are now fixed, this is equivalent to maximizing the following function of the~$\flr_x$:
\begin{equation}
F(\flr) \,=\, \prod_{\{x,y\}}\,
\frac{{\flr_x}^{v_{xy}}\,{\flr_y}^{v_{yx}}}{(\flr_x+\flr_y)^{t_{xy}}},
\label{eq:funcioF}
\end{equation}
(recall that $v_{xy} = V_{xy}/V$ and $t_{xy} = T_{xy}/V$ where $V$~is a positive constant greater than or equal to any of the turnouts $T_{xy}$; going from (\ref{eq:probabilitat}) to (\ref{eq:funcioF}) involves taking the power of exponent $1/V$ and disregarding a fixed multiplicative constant).
\ensep
The function $F$ is certainly smooth on $\phiset$. Besides, it~is clearly bounded from above, since
it is a product of several factors less than or equal to~$1$. However, generally speaking $F$~needs not to achieve a maximum in $\phiset$, because this set is not compact. In~the present situation, the~only general fact that one can guarantee in this connection is the existence of \dfd{maximizing sequences}, \ie sequences $\flrn$ in~$\phiset$ with the property that $F(\flrn)$ converges to the lowest upper bound $\lub=\sup\,\{F(\psi)\,|\,\psi\in\phiset\}$.

\vskip4pt
\medskip
The theorems of this section collect the basic results that we need about Zermelo's method.
The first theorem is standard except for part~(c). However, we prove also parts (a) and (b) because in so doing we introduce several ideas and techniques that are used in part~(c) and in other parts of the paper.

\vskip-6mm\null
\begin{theorem}[Zermelo, 1929 \cite{ze}; see also \cite{fo,ke}]\hskip.5em
\label{st:zermeloirre}
If $\vaa$ is irreducible,\, then:

\iim{a}There is a unique $\flr\in\phiset$ which maximizes $F$ on $\phiset$.

\iim{b}$\flr$ is the solution of the following system of equations:
\begin{align}
\sum_{y\neq x}\,t_{xy}\,\frac{\flr_x}{\flr_x+\flr_y} \,&=\, \sum_{y\neq x}\,v_{xy},
\label{eq:fratesZ}
\\[2.5pt]
\sum_x\,\flr_x \,&=\, 1,
\label{eq:fratesaZ}
\end{align}
\vskip-8pt\hskip1.5\parindent
where $(\ref{eq:fratesZ})$~contains one equation for every~$x$.

\iim{c}$\flr$ is an infinitely differentiable function of the scores $v_{xy}$ as long as they keep satisfying the hypothesis of irreducibility.
\end{theorem}

\setlength\repskip{1.05em} 

\begin{proof}\hskip.5em
Let us begin by noticing that the hypothesis of irreducibility entails that $F$ can be extended to a continuous function on $\phisetb$ by putting $F(\psi)=0$ for $\psi\in\partial\phiset$. In order to prove this claim we must show that $F(\psi^n)\rightarrow 0$ whenever $\psi^n$ converges to a point $\psi\in\partial\phiset$. Let us consider the following sets associated with $\psi$: $\xst=\{x\,|\,\psi_x > 0\}$
and $\yst=\{y\,|\,\psi_y = 0\}$. The second one is not empty since we are assuming $\psi\in\partial\phiset$, whereas the first one is not empty because the strengths add up to the positive value~$1$. Now, for any $x\in\xst$ and $y\in\yst$, $F(\psi^n)$
contains a factor of the form $(\psi^n_y)^{v_{yx}}$, which tends to zero as soon as~$v_{yx}>0$
(while the other factors remain bounded).
So, the only way for $F(\psi^n)$ not to approach zero would be $\vyx=\zeromatrix$, in contradiction with the irreducibility of $\vaa$.

After such an extension, $F$ is a continuous function on the compact set $\phisetb$.\linebreak 
So, there exists $\flr$ which maximizes $F$ on $\phisetb$. However, since $F(\psi)$ vanishes on $\partial\phiset$ whereas it is strictly positive for $\psi\in\phiset$, any maximizer $\flr$ must belong to $\phiset$.
This establishes the existence part of~(a).

Since $F$ is constant on every ray from the origin,
maximizing it on $\phiset$ amounts to the same thing as
maximizing it on the positive orthant $\bbr_+^\ist$.
On~the other hand, maximizing $F$ is certainly equivalent to maximizing $\log F$.
Now, a maximizer of $\log F$ on~$\bbr_+^\ist$ must satisfy the differential conditions
\begin{equation}
\label{eq:gradient}
\frac{\partial\log F(\flr)}{\partial\flr_x} \,=\,
\sum_{y\neq x}\left(\frac{v_{xy}}{\flr_x} - \frac{t_{xy}}{\flr_x+\flr_y}\right) \,=\, 0,
\end{equation}
where $x$~varies over~$\ist$.
Multiplying each of these equations by the corresponding $\flr_x$ results in the system of equations~(\ref{eq:fratesZ}).
\ensep
That system contains $N$~equations for the $N$~variables $\flr_x$ ($x\in\ist$); however, it is redundant: by~using the fact that $v_{xy}+v_{yx} = t_{xy}$, one easily sees that adding up all of the equations in (\ref{eq:fratesZ}) results in a tautology.
\ensep
That is why one can supplement that system with equation~(\ref{eq:fratesaZ}),
which selects the maximizer in $\phiset$.

\emph{Let us see now that the maximizer is unique.}
Instead of following the interesting proof given by Zermelo,
here we will prefer to follow~\cite{ke}, which will have the advantage of preparing matters for part~(c). More specifically, the uniqueness will be obtained by seeing that any critical point of \,$\log F$ as a function on $\phiset$, \ie any solution of (\ref{eq:fratesZ}--\ref{eq:fratesaZ}), is a strict local maximum; this~implies that there is only one critical point, because otherwise
one should have other kinds of critical points~\cite[\S VI.6]{cou}
(we are invoking the so-called mountain pass theorem;
here we are using the fact that $\log F$ tends to $-\infty$ as $\flr$ approaches $\partial\phiset$).
In order to study the character of a critical point we will look at the second derivatives of \,$\log F$ with respect to $\flr$.
By~differentiating~(\ref{eq:gradient}), one obtains that
\begin{align}
\label{eq:hessianxx}
\frac{\partial^2\log F(\flr)}{\partial\flr_x{}^2} \,&=\,
-\, \sum_{y\neq x} \left(\frac{v_{xy}}{\flr^2_x} \,-\,
\frac{t_{xy}}{(\flr_x+\flr_y)^2}\right),
\\[2.5pt]
\label{eq:hessianxy}
\frac{\partial^2\log F(\flr)}{\partial\flr_x\,\partial\flr_y} \,&=\,
\frac{t_{xy}}{(\flr_x+\flr_y)^2},\qquad\text{ for }x\neq y.
\end{align}
On the other hand, when $\flr$ is a critical point, equation~(\ref{eq:fratesZ}) transforms (\ref{eq:hessianxx}) into the following expression:
\begin{equation}
\label{eq:hessianxxBis}
\frac{\partial^2\log F(\flr)}{\partial\flr_x{}^2} \,=\,
-\, \sum_{y\neq x}\,\frac{t_{xy}}{(\flr_x+\flr_y)^2}\,\frac{\flr_{y}}{\flr_{x}}.
\end{equation}

So, the Hessian quadratic form is as follows:
\begin{equation}
\label{eq:hessian}
\begin{split}
\sum_{x,y}\left(
\frac{\partial^2\log F(\flr)}{\partial\flr_x\,\partial\flr_y}
\right)\,\psi_x\,\psi_y
\,&=\, - \sum_{x,y\neq x}\,\frac{t_{xy}}{(\flr_x+\flr_y)^2}\,
\left(\frac{\flr_{y}}{\flr_x} \psi_x^2 - \psi_x\psi_y\right)\\
\,=\, - \sum_{x,y\neq x}\,&\frac{t_{xy}}{(\flr_x+\flr_y)^2\,\flr_x\flr_y}\,
\left(\flr_y^2\psi_x^2 - \flr_x\flr_y\psi_x\psi_y\right)\\
\,=\, -\, \sum_{\{x,y\}}\,&\frac{t_{xy}}{(\flr_x+\flr_y)^2\,\flr_x\flr_y}\,
\left(\flr_y\psi_x - \flr_x\psi_y\right)^2,
\end{split}
\end{equation}
where the last sum runs through all unordered pairs $\{x,y\}\sbseteq\ist$ with $x\neq y$.
The last expression is non-positive and it vanishes \ifoi
$\psi_x/\flr_x = \psi_y/\flr_y$ for any $x,y\in\ist$
(the ``only~if'' part is immediate when $t_{xy}>0$; for~arbitrary $x$ and $y$, the hypothesis of irreducibility allows to connect them through a path $x_0x_1\dots x_n$ ($x_0=x$, $x_n=y$) with the property that\linebreak 
$t_{x_ix_{i+1}}\ge v_{x_ix_{i+1}}>0$ for~any~$i$, so that one gets $\psi_x/\flr_x = \psi_{x_1}/\flr_{x_1} = \dots = \psi_y/\flr_y$).
So,~the vanishing of (\ref{eq:hessian}) happens \ifoi
$\psi=\lambda\flr$ for some scalar $\lambda$.
\ensep
However, when $\psi$ is restricted to variations such that $\flr+\psi$ stays in~$\phiset$, \ie to 
vectors $\psi\in\bbr^\ist$ satisfying $\sum_x\psi_x=0$,
the case $\psi=\lambda\flr$ reduces to $\lambda=0$ and therefore $\psi=0$ (since $\sum_x\flr_x$ is~positive). So, the Hessian is negative definite 
when restricted to such variations.
This ensures that $\flr$ is a~strict local maximum of \,$\log F$ as a function on~$\phiset$.
In fact, one easily arrives at such a conclusion when Taylor's formula is used to analyse the behaviour of $\log F(\flr+\psi)$ for small $\psi$
satisfying $\sum_x\psi_x=0$.


\textit{Finally, let us consider
part}~(c),\ensep that is,
the dependence of $\flr\in\phiset$ on the matrix~$\vaa$.
To~begin with, we notice that the set $\irreopen$ of irreducible matrices is open since it is a finite intersection of open sets, namely one open set for each splitting of $\ist$ into two sets $\xst$ and $\yst$.
The dependence of $\flr\in\phiset$ on~$\vaa$\linebreak[3] is due to the presence of $v_{xy}$ and $t_{xy}\cd=v_{xy}\cd+v_{yx}$ in the equations~(\ref{eq:fratesZ}--\ref{eq:fratesaZ})\linebreak 
which determine~$\flr$. However, we are not in the standard setting of the implicit function theorem, since we are dealing with a system of $N\cd+1$~equations whilst $\flr$ varies in a~space of dimension~$N\cd-1$.
In order to place oneself in a standard setting, it~is~convenient here to replace the condition of normalization $\sum_x\flr_x=1$ \,by~the alternative one\, $\flr_a=1$, where $a$ is a fixed element of $\ist$.
This change of normalization corresponds to mapping $\phiset$ \,to\, $U = \{\,\flr\in\bbr^\ist\,|\,\flr_x>0\text{ for all }x\in\ist,\, \flr_a=1\,\}$ by means of the diffeo\-morphism
$g:\flr\mapsto\flr/\flr_a$, which has the property that~$F(g(\flr)) = F(\flr)$.
By~an~argument of the same kind as that used at the end of the preceding paragraph,
one sees that the Hessian bilinear form of $\log F$ is negative definite when restricted to variations so as to stay in~$U$.
Therefore, if we take as coordinates on~$U$ the $\flr_x$ with $x\in\ist\setminus\{a\}=:\ist'$,\, the function $F$ restricted to~$U$ has the property that the matrix $(\,{\partial^2\log F(\flr)}/{\partial\flr_x\partial\flr_y}\mid x,y\cd\in\ist')$ is negative definite and therefore invertible, which entails that the system of equations $(\,{\partial\log F(\flr,\vaa)}/{\partial\flr_x}=0\mid x\cd\in\ist')$ 
---where we made explicit the dependence on $\vaa$---
determines $\flr\in U$ as a~smooth function of $\vaa\in\irreopen$.
\end{proof}

\bigskip
The next theorem is the core result for ensuring at the same time both the condition of unanimous decomposition and the continuity of the ratings with respect to the data.
\ensep
Let us recall that a maximizing sequence 
means a sequence $\flrn\in\phiset$ such that $F(\flrn)$ approaches the lowest upper bound of $F$ on $\phiset$.

\smallskip
\begin{theorem}[Statements (a) and (b) are contained in \cite{ze}]
\label{st:zermelore}
Assume that there exists
a top dominant irreducible component $\xst$.
In this case:

\iim{a} There is a unique $\flr\in\phisetb$ such that
any maximizing sequence converges to $\flr$.

\iim{b} $\flrax = 0$, \,whereas $\flrx$ has all components positive and coincides with the solution of a system analogous to \textup{(\ref{eq:fratesZ}--\ref{eq:fratesaZ})} where $x$ and $y$ vary only within $\xst$.

\iim{c} $\flr$ is a continuous function of the scores $v_{xy}$ as long as they keep satisfying the hypotheses of the present theorem.
\end{theorem}

\vskip-2mm 
\remark The below given proof of statements~(a) and (b) follows \brwrap{\dbibref{ze}{p.\,440--442}}. 
Again, we include it because it prepares the path for the proof of~(c). Partial results related to~(c) are contained in \cite[Thm.\,1.1]{cg}. However, they consider only a special way of varying the scores $v_{xy},$ namely adding a common $\varepsilon\downarrow0$ to all the scores. Besides, their proof uses some tools from algebraic geometry, whereas ours stays in the domain of calculus.

\begin{proof}\hskip.5em
The definition of the lowest upper bound immediately implies the existence of maximizing sequences. On the other hand, the compactness of~$\phisetb$ guarantees that any maximizing sequence has a subsequence which converges in~$\phisetb$.
Let $\flrn$ and $\flr$ denote respectively one of such convergent maximizing sequences and its limit.
In the following we will see that $\flr$ must be the unique point specified in statement~(b). This~entails that any maximizing sequence converges itself to $\flr$ (without extracting a subsequence).

\textit{So, our aim is now statement}~(b).\ensep
From now on we will write $\yst=\ist\setminus\xst$,
and a general element of $\bbr_+^\ist$ will be denoted by $\gflr$.
For convenience, in this part of the proof 
we will replace the condition $\sum_x\gflr_x = 1$ by $\sum_x\gflr_x \le 1$ (and similarly for $\flrn$ and $\flr$); since $F(\lambda\gflr) = F(\gflr)$ for any $\lambda>0$,
the properties that we will obtain will be easily translated from
$\widehat\phiset = \{\,\gflr\in\bbr^\ist\mid\gflr_x>0 
\text{ for all }x\in\ist,\ \sum_{x\in\ist}\gflr_x \le 1\,\}$ \,to\, $\phiset$.
On the other hand, it will also be convenient to consider first the case where $\yst$ is also an irreducible component.
In such a case, it is interesting to  rewrite $F(\gflr)$ as a product of three factors:
\begin{equation}
F(\gflr) \,=\, \fxx(\gflrx)\,\fyy(\gflry)\,\fxy(\gflrx, \gflry),
\end{equation}
namely:
{\allowdisplaybreaks 
\begin{align}
\fxx(\gflrx) \,\,&=\, \prod_{\{x,\bar x\}\sbseteq\xst}\,
\frac{{\gflr_x}^{v_{x\bar x}}\,{\gflr_{\bar x}}^{v_{\bar xx}}}{(\gflr_x+\gflr_{\bar x})^{t_{x\bar x}}},
\label{eq:funcioFxx}
\\[2.5pt]
\fyy(\gflry) \,\,&=\, \prod_{\{y,\bar y\}\sbseteq\yst}\,
\frac{{\gflr_y}^{v_{y\bar y}}\,{\gflr_{\bar y}}^{v_{\bar yy}}}{(\gflr_y+\gflr_{\bar y})^{t_{y\bar y}}},
\label{eq:funcioFyy}
\\[2.5pt]
\fxy(\gflrx,\gflry) \,\,&=\,
\kern6pt\prod_{{\scriptstyle x\in\xst \atop\scriptstyle y\in\yst }}\,\kern6pt
\left(\frac{\gflr_x}{\gflr_x+\gflr_y}\right)^{\hskip-3ptv_{xy}}\hskip-3pt,
\label{eq:funcioFxy}
\end{align}
}
where we used that $v_{yx}=0$ and $t_{xy}=v_{xy}$. Now, let us look at the effect of replacing $\gflry$ by $\lambda\gflry$ without varying $\gflrx$. The values of $\fxx$ and $\fyy$ remain unchanged, but that of $\fxy$ varies in the following way:
\begin{equation}
\frac{\fxy(\gflrx,\lambda\gflry)}{\fxy(\gflrx,\gflry)} \,\,=\,
\kern6pt\prod_{{\scriptstyle x\in\xst \atop\scriptstyle y\in\yst }}\,\kern6pt
\left(\frac{\gflr_x+\gflr_y}{\gflr_x+\lambda\gflr_y}\right)^{\hskip-3ptv_{xy}}.
\label{eq:factorFxy}
\end{equation}
In particular, for $0<\lambda<1$ each of the factors of the right-hand side of~(\ref{eq:factorFxy}) is greater than or equal to $1$. This remark leads to the following argument.
\ensep
First, we~can see that $\flrn_y/\flrn_x\rightarrow 0$ for any $x\in\xst$ and $y\in\yst$ such that $v_{xy}>0$ (such pairs $xy$ exist because of the hypothesis that $\xst$ dominates~$\yst$). Otherwise, the~preceding remark entails that the sequence $\flrbis^n = (\flrx^n,\lambda\flry^n)$ with $0<\lambda<1$ would satisfy \,$F(\flrbis^n) > K F(\flrn)$ for some $K>1$ and infinitely many $n$, in contradiction with the hypothesis that $\flrn$ was a maximizing sequence. On the other hand, we see also that $\fxy(\flrn)$ approaches its lowest upper bound, namely~$1$.
\ensep
Having achieved such a property, the problem of maximizing $F$ reduces to separately maximizing $\fxx$ and $\fyy$, which is solved by Theorem~\ref{st:zermeloirre}. For~the moment we are dealing with relative strengths only, \ie without any normalizing condition like~(\ref{eq:fratesaZ}). So, we see that $\fyy$ gets optimized when each of the ratios $\flrn_y/\flrn_{\bar y}\ (y,\bar y\in\yst)$ approaches the homologous one for the unique maximizer of~$\fyy$, and analogously with~$\fxx$. Since these ratios are finite positive quantities, the statement that $\flrn_y/\flrn_x\rightarrow 0$ becomes extended to any $x\in\xst$ and $y\in\yst$ whatsoever (since one can write $\flrn_y/\flrn_x = (\flrn_y/\flrn_{\bar y})\times(\flrn_{\bar y}/\flrn_{\bar x})\times(\flrn_{\bar x}/\flrn_x)$ with $v_{\bar x\bar y}> 0$). Let us recover now the condition $\sum_{x\in\ist}\flrn_x=1$. The preceding facts imply that $\flry^n\rightarrow 0$, whereas $\flrx^n$ converges to the unique maximizer of $\fxx$. This establishes~(b) as well as the uniqueness part of~(a).

The general case where $\yst$ decomposes into several irreducible components, all of them dominated by $\xst$, can be taken care of by
induction over the different irreducible components of $\ist$.
At each step, one deals with an irreducible component $\zst$ with the property of being minimal, in the sense of the dominance relation~$\chrela$, among those which are still pending. By~means of an argument analogous to that of the preceding paragraph, one sees that:\ensep
(i)~$\flrn_z/\flrn_x\rightarrow 0$ for any $z\in\zst$~and~$x$ such that $x\chrela z$ with $v_{xz}>0$;\ensep
(ii)~the ratios $\flrn_z/\flrn_{\bar z}\ (z,\bar z\in\zst)$ approach the homologous ones for the unique maximizer of~$\fzz$;\ensep
and (iii)~$\flrr^n$ is a maximizing sequence for $\frr$, where $R$~denotes the union of the pending components, $\zst$ excluded.
Once this induction process has been completed, one can combine its partial results to show that $\flrn_z/\flrn_x\rightarrow 0$ for any $x\in\xst$ and $z\not\in\xst$ (it suffices to consider a path $x_0x_1\dots x_n$ from $x_0\in\xst$ to $x_n=z$ with the property that $v_{x_ix_{i+1}}>0$ for any $i$ and to notice that each of the factors $\flrn_{x_{i+1}}/\flrn_{x_i}$ remains bounded while at least one of them tends to zero). As above, one concludes that $\flrax^n\rightarrow 0$, whereas $\flrx^n$ converges to the unique maximizer of $\fxx$.

\smallskip
\textit{The two following remarks will be useful in the proof of part}~(c):
\linebreak 
(1)~\label{remarkOne}According to the proof above, $\flrx$~is determined (up to a multiplicative constant) by equations~(\ref{eq:fratesZ})
with $x$ and $y$ varying only within~$\xst$:
\smallskip\null
\begin{equation}
\label{eq:yinx}
\fun_x(\flrx,\vaa) \,:=\,
\sum_{\latop{\scriptstyle y\in\xst}{\scriptstyle y\neq x}}\,
t_{xy}\,\frac{\flr_x}{\flr_x+\flr_y} \,-\,
\sum_{\latop{\scriptstyle y\in\xst}{\scriptstyle y\neq x}}\,
v_{xy} \,=\, 0,\qquad\forall x\in\xst.\kern-10pt
\end{equation}
However, since $y\in\ist\setminus\xst$ implies on the one hand $\flr_y=0$ and on the other hand $t_{xy}=v_{xy}$, each of the preceding equations is equivalent to a similar one where $y$ varies over the whole of $\ist\setminus\{x\}$:
\begin{equation}
\label{eq:ally}
\funbis_x(\flr,\vaa) \,:=\,
\sum_{\latop{\scriptstyle y\in\ist}{\scriptstyle y\neq x}}\,
t_{xy}\,\frac{\flr_x}{\flr_x+\flr_y} \,-\,
\sum_{\latop{\scriptstyle y\in\ist}{\scriptstyle y\neq x}}\,
v_{xy} \,=\, 0,\qquad\forall x\in\xst.\kern-10pt
\end{equation}
(2)~\label{remarkTwo}Also, it is interesting to see the result of
adding up the equations (\ref{eq:ally}) for all $x$ in some subset $\wst$ of $\xst$. Using the fact that $v_{xy}+v_{yx}=t_{xy}$, one sees that such an addition results in the following equality:
\begin{equation}
\label{eq:sumw}
\sum_{\latop{\scriptstyle x\in\wst}{\scriptstyle y\not\in\wst}}\,
t_{xy}\,\frac{\flr_x}{\flr_x+\flr_y} \,-\,
\sum_{\latop{\scriptstyle x\in\wst}{\scriptstyle y\not\in\wst}}\,
v_{xy} \,=\, 0,\qquad\forall\kern.75pt\wst\sbseteq\xst.\kern-5pt
\end{equation}

\textit{Let us proceed now with the proof of}~(c).\ensep
In the following, $\vaa$ and $\vaabis$\linebreak[3] denote respectively a fixed matrix satisfying the hypotheses of the theorem and a~slight perturbation of it.
In the following we systematically use a tilde to distinguish between hom\-olo\-gous objects associated respectively with $\vaa$~and~$\vaabis$; in particular, such a notation will be used in connection with the labels of certain equations.
Our aim is to~show that\, $\flrbis$ approaches $\flr$ \,as\, $\vaabis$ approaches $\vaa$. In this connection we will use the little-o and big-O notations made popular by Edmund Landau (who, by the way, wrote also on the rating of chess players~\cite{la:1895,la:1914}, as we will see in \secpar{5.1}).\ensep
This notation refers here to
functions of $\vaabis$ and their behaviour as $\vaabis$ approaches $\vaa$;
\ensep
if $f$ and $g$ are two such functions,
\ensep
$f=o(g)$\,~means that for every $\epsilon>0$ there exists a~$\delta>0$ such that $\|\vaabis-\vaa\|\le\delta$ implies $\|f(\vaabis)\|\le\epsilon\,\|g(\vaabis)\|$;
\ensep
on~the~other hand, 
$f=O(g)$\, means that there exist $M$ and~$\delta>0$ such that $\|\vaabis-\vaa\|\le\delta$ implies $\|f(\vaabis)\|\le M\,\|g(\vaabis)\|$.

Obviously, if $\vaabis$ is near enough to $\vaa$ then $v_{xy}>0$ implies $\vbis_{xy}>0$.\linebreak 
As a consequence, $x\chrel y$ implies $x\chrelbis y$. In particular, the irreducibility
of~$\vxx$ entails that $\vaabisxx$~is also irreducible.
Therefore, $\xst$ is entirely contained in~some irreducible component $\xstbis$ of $\ist$ for $\vaabis$.
Besides, $\xstbis$ is a top dominant irreducible component for $\vaabis$; in fact, we have the following chain of implications for $x\in\xst\sbseteq\xstbis$: $y\not\in\xstbis \,\Rightarrow\, y\not\in\xst \,\Rightarrow\, x\chrela y \,\Rightarrow\, x\chrelbis y \,\Rightarrow\, x\chrelabis y$, where we have used successively: the inclusion $\xst\sbseteq\xstbis$, the hypothesis that $\xst$ is top dominant for $\vaa$, the fact that $\vaabis$ is near enough to $\vaa$, and the hypothesis that $y$ does not belong to the irreducible component $\xstbis$.\linebreak 
Now, according to part~(b) and remark~(1) from p.\,\pageref{remarkOne}--\pageref{remarkTwo},
$\flrx$ and $\flrxbis$ are~determined respectively by the systems (\ref{eq:yinx}) and (\tref{eq:yinx}), or equivalently by~(\ref{eq:ally}) and (\tref{eq:ally}), whereas $\flrax$ and $\flraxbis$ are both of them equal to zero. So~we~must show that $\flrbisy=o(1)$ for any $y\in\xstbis\setminus\xst$, and that $\flrbisx-\flr_x=o(1)$ for~any $x\in\xst$. The proof is organized in three main steps.

\halfsmallskip
\textit{Step}~(1).\ensep
\textit{$\flrbisy=O(\flrbisx)$ whenever $v_{xy}>0$}.\ensep
For the moment, we assume $\vaabis$ fixed (near enough to $\vaa$ so that $\vbis_{xy}>0$) and $x,y\in\xstbis$.
Under these hypotheses one can argue as follows:
Since $\flrxbis$ maximizes $\fbisxx$, the corresponding value of $\fbisxx$ can be bounded from below by any particular value of the same function. On~the other hand, we can bound it from above by the factor $(\flrbisx/(\flrbisx+\flrbisy))^{\vbis_{xy}}$. So, we can write
\begin{equation}
\label{eq:desigualtatsstep1}
\kern-10pt
\left(\frac12\right)^{\hskip-3pt\textstyle\frac{N(N-1)}{2}} \kern-2pt\le \left(\frac12\right)^{\lower2pt\hbox{$\latop{\textstyle\sum}{\vrule width0pt height8pt\scriptstyle \{p,q\}\sbseteq\xstbiss}$}\kern-2pt\raise4pt\hbox{$\textstyle \tbis_{pq}$}} =
\fbisxx(\psi)
\,\le\, \fbisxx(\flrxbis) \,\le\,
\left(\frac{\flrbisx}{\flrbisx+\flrbisy}\right)^{\hskip-3pt\vbis_{xy}}\kern-3pt,\kern-2pt
\end{equation}
where $\psi$ has been taken so that $\psi_q$ has the same value for all $q\in\xstbis$.
The preceding inequality entails that
\begin{equation}
\label{eq:desigualtatsstep1bis}
\flrbisy \,\le\, \left(2^{\,N(N-1)\,/\,\vbis_{xy}} - 1\right) \,\flrbisx.
\end{equation}
Now, this inequality holds not only for $x,y\in\xstbis$,
but it is also trivially true for~$y\not\in\xstbis$,
since then one has $\flrbisy=0$.
On the other hand, the case $y\in\xstbis,\ x\not\in\xstbis$ is not possible at all, because the hypothesis that
$\vbis_{xy}>0$ would then contradict the fact that $\xstbis$ is a top dominant irreducible component.
\ensep
Finally, we let $\vaabis$ vary towards $\vaa$. The~desired result is a consequence of~(\ref{eq:desigualtatsstep1bis}) since $\vbis_{xy}$ approaches $v_{xy}>0$.

\halfsmallskip
\textit{Step}~(2).\ensep
\textit{$\flrbisy=o(\flrbisx)$ for any $x\in\xst$ and $y\not\in\xst$}.\ensep
Again, we will consider first the special case where $v_{xy}>0$.
In this case the result is easily obtained as a consequence of the equality (\tref{eq:sumw}) for $\wst=\xst$:
\begin{equation}
\label{eq:sumx}
\sum_{\latop{\scriptstyle x\in\xst}{\scriptstyle y\not\in\xst}}\,
\tbis_{xy}\,\frac{\flrbisx}{\flrbisx+\flrbisy} \,-\,
\sum_{\latop{\scriptstyle x\in\xst}{\scriptstyle y\not\in\xst}}\,
\vbis_{xy} \,=\, 0.
\end{equation}
In fact, this equality implies that
\begin{equation}
\label{eq:sumxbis}
\sum_{\latop{\scriptstyle x\in\xst}{\scriptstyle y\not\in\xst}}\,
\tbis_{xy}\,\left(1-\frac{\flrbisx}{\flrbisx+\flrbisy}\right) \,=\,
\sum_{\latop{\scriptstyle x\in\xst}{\scriptstyle y\not\in\xst}}\,
\vbis_{yx}.
\end{equation}
Now, it is clear that the right-hand side of this equation is $o(1)$ and that each of the terms of the left-hand side is positive or zero. Since $\tbis_{xy} - v_{xy} = \tbis_{xy} - t_{xy} = o(1)$, the hypothesis that $v_{xy}>0$ allows to conclude that $\flrbisx/(\flrbisx+\flrbisy)$ approaches~$1$, or equivalently, $\flrbisy=o(\flrbisx)$.
\ensep
Let us consider now the case of any $x\in\xst$ and $y\not\in\xst$. Since $\xst$ is top dominant, we know that there exists a path $x_0x_1\dots x_n$ from $x_0=x$ to $x_n=y$ such that $v_{x_ix_{i+1}}>0$ for all $i$. According to step~(1) we have $\flrbissub_{x_{i+1}}=O(\flrbissub_{x_i})$. On the other hand, there must be some $j$ such that $x_j\in\xst$ but $x_{j+1}\not\in\xst$, which has been seen to imply that $\flrbissub_{x_{j+1}}=o(\flrbissub_{x_j})$. By combining these facts one obtains the desired result.

\halfsmallskip
\textit{Step}~(3).\ensep
\textit{$\flrbisx-\flr_x=o(1)$ for~any $x\in\xst$}.\ensep
Consider the equations (\tref{eq:ally}) for $x\in\xst$ and split the sums in two parts depending on whether $y\in\xst$ or $y\not\in\xst$:
\begin{equation}
\label{eq:arranged}
\sum_{\latop{\scriptstyle y\in\xst}{\scriptstyle y\neq x}}\,
\tbis_{xy}\,\frac{\flrbisx}{\flrbisx+\flrbisy} \,-\,
\sum_{\latop{\scriptstyle y\in\xst}{\scriptstyle y\neq x}}\,
\vbis_{xy} \,=\,
\sum_{y\not\in\xst}\,(\vbis_{xy}-\tbis_{xy}\,\frac{\flrbisx}{\flrbisx+\flrbisy}).
\end{equation}
The last sum is $o(1)$ since step~(2) ensures that $\flrbisy=o(\flrbisx)$ and we also know that $\tbis_{xy}-\vbis_{xy}=\vbis_{yx} = o(1)$ (because $x\in\xst$ and $y\not\in\xst$). So $\flrbis$ satisfies a system of the following form, where $x$ and $y$ vary only within~$\xst$ \,and\, $\widetilde w_{xy}$ is a slight modification of $\vbis_{xy}$ which absorbs the right-hand side of (\ref{eq:arranged}):
\begin{equation}
\label{eq:vwdetall}
\funter_x(\flrbisX,\vaabis,\waabis) \,:=\,
\sum_{\latop{\scriptstyle y\in\xst}{\scriptstyle y\neq x}}\,
\tbis_{xy}\,\frac{\flrbisx}{\flrbisx+\flrbisy} \,-\,
\sum_{\latop{\scriptstyle y\in\xst}{\scriptstyle y\neq x}}\,
\widetilde w_{xy} \,=\, 0,\hskip1.5em\forall x\in\xst.\kern-5pt
\end{equation}
Here, the second argument of $\funter$ refers to the dependence on $\vaabis$ through~$\tbis_{xy}$.
We know that $\tbis_{xy} - t_{xy} = o(1)$ and also that $\widetilde w_{xy} - v_{xy} = (\widetilde w_{xy} - \vbis_{xy}) +\linebreak 
(\vbis_{xy} - v_{xy}) = o(1)$.
So~we are interested in the preceding equation near the point $(\flrx,\vaa,\vaa)$. Now, in this point we have $\funter(\flrx,\vaa,\vaa) = \fun(\flrx,\vaa) = 0$, as~well~as $(\partial\funter_x/\partial\flrbisy)(\flrx,\vaa,\vaa) = (\partial\fun_x/\partial\flr_y)(\flrx,\vaa)$. Therefore, the implicit function theorem can be applied similarly as in Theorem~\ref{st:zermeloirre},
with the result that $\flrbisX=\funiv(\vaabis,\waabis)$, where $\funiv$ is a smooth function that satisfies \hbox{$\funiv(\vaa,\vaa)=\flrx$.} In particular, the continuity of $\funiv$ allows to conclude that $\flrbisX$~approaches $\flrx$, since we know that both $\vaabis$ and $\waabis$ approach $\vaa$.

\halfsmallskip
Finally, by combining the results of steps~(2) and (3) one obtains $\flrbisy\cd=o(1)$ for any $y\not\in\xst$.
\end{proof}

\remarks

1. Part~(a) states that every maximizing sequence converges towards a particular $\phi$ in $\phisetb$.
The converse statement is false: converging towards this $\phi$ is not a sufficient condition for being a maximizing sequence.
The preceding proof shows that a~necessary and sufficient condition for $\flrn$ to be a maximizing sequence is that the ratios $\flrn_z/\flrn_y$ tend to $0$ whenever $y\chrela z$, whereas, for $y\chrels z$, \ie if $y$ and $z$ belong to the same irreducible component $Z$, they approach the homologous ratios for the unique maximizer of~$\fzz$.

2. If there is not a dominant component, then the maximizing sequences can have multiple limit points.

3. The non-linear system (\ref{eq:fratesZ}--\ref{eq:fratesaZ}) can be solved by the following iterative scheme \cite{ze,fo}:
\begin{align}
\sum_{y\neq x}\,t_{xy}\,\frac{\flr^{(n+1)}_x}{\flr^{(n)}_x+\flr^{(n)}_y} \,&=\, \sum_{y\neq x}\,v_{xy},
\label{eq:fratesZn}
\\[2.5pt]
\sum_x\,\flr^{(n+1)}_x \,&=\, 1,
\label{eq:fratesaZn}
\end{align}

\section{CLC structure}
\label{sec:clc}

This section is devoted to paired-comparison matrices with a certain special structure,
namely the structure that arises from the CLC~projection that we introduced in \cite{crc,cri}.
As we will see, these matrices have good properties in connection with Zermelo's method
and with the dominance relation that was defined in~\secpar{1.3}.

\paragraph{3.1}
A paired-comparison matrix will be said to have \dfd{CLC~structure},
or to be a \dfd{CLC~matrix}, when 
\textit{there exists a total order $\xi$ on $\ist$ such that
\begin{alignat}{2}
\label{eq:vxyinequality}
&v_{xy} \,\ge\, v_{yx},
\qquad &&\hbox{whenever $x\rxi y$},
\\[2.5pt]
\label{eq:vequaltomax}
&v_{xz} \,\,=\,\, \max\,(v_{xy},v_{yz}),\qquad &&\hbox{whenever $x\rxi y\rxi z$},
\\[2.5pt]
\label{eq:vequaltomin}
&v_{zx} \,\,=\,\, \min\,(v_{zy},v_{yx}),\qquad &&\hbox{whenever $x\rxi y\rxi z$},
\\[2.5pt]
\label{eq:ttm}
&0 \,\le\, t_{xz} - t_{x'z} \,\le\, m_{xx'},\qquad &&\hbox{whenever $z\not\in\{x,x'\}$},
\end{alignat}
where $x\rxi y$ means that $x$ precedes $y$ in the order $\xi$, and $x'$ denotes the element of $\ist$ that immediately follows $x$ in the order~$\xi$}.
In such a situation 
the total order~$\xi$ will be called an \dfc{admissible order} for the matrix $(v_{xy})$.

\medskip
Our interest in the CLC structure derives from the following fact:

\newcommand\bla{\cite[Thm.~4.5]{cri}}
\begin{theorem}[\bla]
\label{st:characterization}
The CLC~projection always results in a CLC~matrix. Besides, a CLC~matrix is invariant by the CLC~projection.
\end{theorem}

For instance, in the case of the Llull matrix (\ref{eq:llull1}) the CLC projection results in the following CLC matrix:
\begin{equation}
\label{eq:llull1-indirectscores}
(V^\pi_{xy}) \,=\, 
\begin{small}
\begin{tabular}{|c|c|c|c|}
\hlinestrut
\labelcell{a}&10&11&11\\
\hlinestrut
8&\labelcell{b}&11&11\\
\hlinestrut
7&7&\labelcell{c}&11\\
\hlinestrut
7&7&7&\labelcell{d}\\
\hline
\end{tabular}
\end{small}
\,.
\end{equation}
Most of this paper ---the only exceptions are the proofs of Proposition~\ref{st:decomposition-proposition} and Theorem~\ref{st:inversion-thm}--- does not depend on the details of the CLC projection procedure, which are given in \cite{crc,cri}.

\medskip
In the following we will also make use of the following facts:

\begin{lemma}\hskip.5em 
\label{st:pattern} A CLC~matrix satisfies the following inequalities:
\begin{alignat}{3}
\label{eq:vinequalities}
&v_{xz} \,\ge\, v_{yz},\quad &&v_{zx} \,\le\, v_{zy},\qquad
&&\hbox{whenever $x\rxi y$ and $z\not\in\{x,y\}$},
\\[2.5pt]
\label{eq:tinequalities}
&t_{xz} \,\ge\, t_{yz},\quad &&t_{zx} \,\ge\, t_{zy},\qquad
&&\hbox{whenever $x\rxi y$ and $z\not\in\{x,y\}$}.
\end{alignat}
\end{lemma}


\begin{proof}\hskip.5em
Let us begin by noticing that it suffices to prove the following particular inequalities:
\begin{alignat}{2}
\label{eq:vinequality}
&v_{zx} \,\le\, v_{zy},\qquad
&&\hbox{whenever $x\rxi y$ and $z\not\in\{x,y\}$},
\\[2.5pt]
\label{eq:tinequality}
&t_{xz} \,\ge\, t_{yz},\qquad
&&\hbox{whenever $x\rxi y$ and $z\not\in\{x,y\}$}.
\end{alignat}
In fact, (\ref{eq:tinequality}) contains both inequalities of (\ref{eq:tinequalities}) since 
$t_{\alpha\beta}=t_{\beta\alpha}$, 
and the first inequality of (\ref{eq:vinequalities}) follows from (\ref{eq:vinequality}) and (\ref{eq:tinequality})
since $t_{\alpha\beta}=v_{\alpha\beta}+v_{\beta\alpha}$.
\ensep
In order to prove (\ref{eq:vinequality}--\ref{eq:tinequality})
we will distinguish three cases:\ensep
(i)~$x\rxi y\rxi z$;\ensep
(ii)~$z\rxi x\rxi y$;\ensep
(iii)~$x\rxi z\rxi y$.

\halfsmallskip Case~(i)\,: $x\rxi y\rxi z$.\ensep
In this case, the inequality (\ref{eq:vinequality}) derives from~(\ref{eq:vequaltomin}).
On the other hand, (\ref{eq:tinequality}) follows by an iterated application of the first inequality of (\ref{eq:ttm}): $t_{xz}\ge t_{x'z}\ge \dots \ge t_{yz}.$

\halfsmallskip Case~(ii)\,: $z\rxi x\rxi y$.\ensep
This case is analogous to the preceding one, with the only difference that it relies on (\ref{eq:vequaltomax}) instead of (\ref{eq:vequaltomin}).

\halfsmallskip Case~(iii)\,: $x\rxi z\rxi y$.\ensep 
In order to deal with this case, we will start by the special subcase where $x,z,y$ are consecutive in the order $\xi$, \ie we will start by the inequalities
\begin{align}
\label{eq:creix-v-horizont-diagonal}
v_{x'x}\,&\le\, v_{x'x''},
\\[2.5pt]
\label{eq:creix-t-diagonal}
t_{xx'}\,&\ge\, t_{x''x'}.
\end{align}
These inequalities are obtained by adding up two particular cases of (\ref{eq:ttm}),
namely the one where $z$ is replaced by $x''$
and the one where $x$ and $z$ are replaced respectively by $x'$ and $x$.
In fact, this addition results in
\begin{equation}
0\,\le\, t_{xx'}-t_{x'x''}\,\le\, m_{x'x''}+m_{xx'},
\label{eq:suma-eqsts}
\end{equation}
whose two inequalities give respectively (\ref{eq:creix-t-diagonal}) and (\ref{eq:creix-v-horizont-diagonal}).
\ensep
Finally, the general situation $x\rxi z\rxi y$ can be dealt with by combining (\ref{eq:creix-t-diagonal}) and (\ref{eq:creix-v-horizont-diagonal}) with the results of cases~(i) and~(ii): In fact, if $a$ denotes the immediate predecessor of $z$ in the order~$\xi$, we can write
\begin{alignat*}{4}
&v_{zx} \,&&\le\, v_{za} \,&&\le\, v_{zz'} \,&&\le\, v_{zy},
\\[2.5pt]
&t_{xz} \,&&\ge\, t_{az} \,&&\ge\, t_{z'z} \,&&\ge\, t_{yz}.\qedhere
\end{alignat*}
\end{proof}

\halfsmallskip
\begin{proposition}\hskip.5em
\label{st:existencia-X} A non-vanishing CLC~matrix has a top dominant irreducible component $\xst$ with the special property that
\begin{equation}
\label{eq:irred-projectada-positivitat}
v_{xy}\,>\,0,\qquad\mbox{whenever\, $x\in\xst$ and\, $y\ne x$.}
\end{equation}
\end{proposition}

\newcommand{\fst}{a}

\begin{proof}\hskip.5em
If $v_{xy}>0$ for all $x,y$, then $(v_{xy})$ is irreducible 
and we are done. So, let us assume that $v_{xy}=0$ for some $x,y$.
By~(\ref{eq:vxyinequality}) and (\ref{eq:vequaltomin}),
this implies the existence of some $p$ such that $v_{p'p}=0$.
Here we are considering an arbitrary admissible order $\xi$,
which we fix for the rest of the proof.
Let $\fst$ be the first element of $\ist$ according to this order.
We will see that the top dominant component is the set $\xst$ defined by
$$
\xst \,=\,
\begin{cases}
\{x\in A \mid v_{p'p}>0\,\mbox{ for all }p\rxi x \},&\text{if $v_{p'p}>0$ for some $p$},\\
\{a\},&\text{if $v_{p'p}=0$ for any $p$}.
\end{cases}
$$
From this definition it immediately follows that having $x\in\xst$ and $y\not\in\xst$ implies $x\rxi y$. This fact will be used repeatedly in the following.

From the definition, it is also clear that for any $x\in\xst$ and $y\not\in\xst$ there
exists $p$ with $x\rxieq p\rxi y$ such that $v_{p'p}=0$. By virtue of 
(\ref{eq:vequaltomin}), it follows that
\begin{equation}
\label{eq:irred-projectada-zeros}
v_{yx}\,=\,0,\qquad\mbox{whenever\, $x\in\xst$ and $y\not\in\xst$.}
\end{equation}
The claim that $\xst$ is the top dominant component will be a consequence of the preceding property together with~(\ref{eq:irred-projectada-positivitat}), to which we devote the rest of the proof.

Let us begin by seeing that $v_{\fst\fst'}>0$.
In fact, according to (\ref{eq:vxyinequality}) having $v_{\fst\fst'}=0$ would imply $v_{\fst'\fst}=0$ and therefore $t_{\fst\fst'}=0$; by (\ref{eq:tinequalities}), this would imply the vanishing of the whole matrix $(v_{xy})$, against one of the assumptions.\ensep
Now, by virtue of (\ref{eq:vinequalities}) it follows that
$v_{\fst y}\,>\,0$\, for all\,$y\ne\fst$.
This finishes the proof if $\xst$ consists of $\fst$ only.
\ensep
In the other cases, observe first that the definition of $\xst$ combined with (\ref{eq:vxyinequality}) and (\ref{eq:vequaltomin}) ensures
$v_{x\bar x}\,>\,0$\, for all\, $x,\bar x\in\xst.$
Finally, (\ref{eq:vinequalities}) allows to derive that
$v_{xy}\,>\,0$\, for all\, $x\in\xst$ and $y\not\in\xst$,
which completes the proof.
\end{proof}

%
%
%
%
%
%
%

\def\auu{a}
\def\cuu{b}
\def\ciu{c}
\def\buu{d}
\def\duu{e}
\def\fuu{f}

\smallskip
\begin{lemma}\hskip.5em
\label{st:obsRosa}
A non-vanishing CLC~matrix has the following properties, where
$\xi$ is any admissible order,
$\rho_x$ are the mean preference scores,
and $\xst$ is the top dominant component:

\iim{\auu} $x\rxi y$ \,implies\ \, $\rho_x\ge\rho_y$.

\iim{\cuu} $\rho_x>\rho_y$ \,implies\ \,
the inequalities
$(\ref{eq:vxyinequality})$ and $(\ref{eq:vinequalities})$.

\iim{\ciu} $\rho_x=\rho_y$ \,implies\ \,
that $(\ref{eq:vxyinequality})$ and $(\ref{eq:vinequalities})$
hold with equality signs.

\iim{\buu} $\rho_x=\rho_y$ \,\ifoi\, $v_{xy}=v_{yx}$.

\iim{\duu} $\rho_x>\rho_y$ \,\ifoi\, $v_{xy}>v_{yx}$.


\iim{\fuu} $\rho_x>\rho_y$ \,whenever $x\in\xst$ and $y\not\in\xst$.
\end{lemma}

\begin{proof}\hskip.5em
Let us recall that the
mean preference scores $\rho_x$ are defined by equation~(\ref{eq:rrates}).
From that equation it follows that
\begin{equation}
\label{eq:rhodif}
(N-1)(\rho_x-\rho_y) \,=\, (v_{xy}-v_{yx})+ \sum_{z\ne x,y} (v_{xz}-v_{yz}).
\end{equation}
In the sequel we will use also the fact that,
according to the definition of CLC~matrix and Lemma~\ref{st:pattern},
\begin{equation}
\label{eq:impin}
x\rxi y \hskip.75em\Rightarrow\hskip.75em
\text{\vtop{\hsize92mm\parindent=0pt
inequalities (\ref{eq:vxyinequality}) and (\ref{eq:vinequalities}),\ensep namely:\hfil\break
$v_{xy}\ge v_{yx},\hskip.5em
v_{xz}\ge v_{yz},\hskip.5em
v_{zx}\le v_{zy},\hskip.5em
\text{for any } z\not\in\{x,y\}.$}}
\end{equation}
Statement~(\auu) is an immediate consequence of combining (\ref{eq:rhodif}) and (\ref{eq:impin}).
\ensep
Since $\xi$ is a total order, the contrapositive of (\auu) amounts to say that
$\rho_x>\rho_y$ implies $x\rxi y$.
Combining this implication with~(\ref{eq:impin}) gives~(\cuu).
\ensep
Let us now assume $\rho_x=\rho_y$ as in~(\ciu);\ensep
since $\xi$ is a total order, we can also assume without loss of generality that $x\rxi y$;\ensep
according to (\ref{eq:impin}), this ensures that all the terms of the right-hand side of (\ref{eq:rhodif})
are positive or zero;\ensep since the left-hand side vanishes, we arrive at the conclusion that every 
term of the right-hand side must vanish; so, we get $v_{xy} = v_{yx}$ as well as
$v_{xz} = v_{yz}$ for any $z\not\in\{x,y\}$. 
In order to complete the proof of~(\ciu) it remains to prove that we have also $v_{zx} = v_{zy}$ for any $z\not\in\{x,y\}$.
This will be a consequence of the fact that we will prove next.

\smallskip
In fact, we claim that
\begin{equation}
\label{eq:impeq}
v_{xy}= v_{yx} \hskip.75em\Rightarrow\hskip.75em
v_{xz}= v_{yz},\hskip.5em
v_{zx}= v_{zy},\hskip.5em
\text{for any } z\not\in\{x,y\}.
\end{equation}
In order to prove this implication we will distinguish three cases:\ensep
\ensep
Case~(i)\,: $x\rxi y\rxi z$.\ensep
In this case it suffices to notice that
\begin{alignat*}{3}
v_{xz}\,&=\,\max(v_{xy},v_{yz})\,&&=\,\max(v_{yx},v_{yz})\,&&=\,v_{yz},
\\[2.5pt]
v_{yz}\,&=\,\min(v_{yx},v_{xz})\,&&=\,\min(v_{xy},v_{xz})\,&&=\,v_{xz}
\end{alignat*}
where we are using successively from left to right: (\ref{eq:vequaltomax}) and  (\ref{eq:vequaltomin}), the assumed equality $v_{xy}=v_{yx}$, 
and (\ref{eq:impin}) with $y$ replaced by $z$.
\ensep
Case~(ii)\,: $z\rxi x\rxi y$.\ensep
This case is analogous to the preceding one, with the difference that $v_{xz}=v_{yz}$ relies on (\ref{eq:vequaltomin}) and $v_{zx}=v_{zy}$ relies on (\ref{eq:vequaltomax}).
\ensep
Case~(iii)\,: $x\rxi z\rxi y$.\ensep 
In this case, (\ref{eq:impin}) allows to write the following inequalities:
\begin{alignat*}{3}
v_{xy}\,&\ge\,v_{xz}\,&&\ge\,v_{yz}\,&&\ge\,v_{yx},
\\[2.5pt]
v_{yx}\,&\le\,v_{zx}\,&&\le\,v_{zy}\,&&\le\,v_{xy}.
\end{alignat*}
When $v_{xy}=v_{yx}$ all of them become equalities, which gives the desired result. This completes the proof of~(\ref{eq:impeq}).

\smallskip
The if part of statement (\buu) relies also on (\ref{eq:impeq}): If $v_{xy}=v_{yx}$, then we have $v_{xz}=v_{yz}$ for any $z\not\in\{x,y\}$,
which results in $\rho_x=\rho_y$ because of (\ref{eq:rhodif}). The only-if part of (\buu) is contained in (\ciu).
\ensep
Concerning statement~(\duu), the implication $\rho_x>\rho_y \,\Rightarrow\, v_{xy}>v_{yx}$ follows easily from~(\cuu) together with~(\buu),
whereas the implication $\rho_x\ge\rho_y \,\Rightarrow\, v_{xy}\ge v_{yx}$ is contained in~(\cuu) together with~(\ciu).
\ensep
Finally, in order to obtain~(\fuu) it suffices to combine (\duu) with Proposition~\ref{st:existencia-X}.
\end{proof}




\paragraph{3.3}
In this paragraph we look at the compatibility between strengths and mean preference scores.

In this connection, Zermelo proved that in the complete (and irreducible) case the strengths always order the options in exactly the same way as the mean preference scores~\cite[\secpar 4]{ze}.

This compatibility easily disappears in the general incomplete case. However, it remains true for CLC~matrices:



\smallskip
\begin{theorem}\hskip.5em
\label{st:phis}
For a CLC~matrix, the associated mean preference scores $\rho_x$ and strengths $\flr_x$ have the following compatibility properties:

\iim{a} $\flr_x > \flr_y \,\Longrightarrow\, \rho_x \,>\, \rho_y$.

\iim{b} $\rho_x \,>\, \rho_y \,\Longrightarrow\, \hbox{either\, }\flr_x > \flr_y \hbox{ \,or\, }\flr_x = \flr_y = 0$.
\end{theorem}

\begin{proof}\hskip.5em
In the following $\xst$ denotes again the top dominant component of the Llull matrix, whose existence has been established by Proposition~\ref{st:existencia-X}. By Theorem~\ref{st:zermelore}, we know that $\flr_x>0$ \ifoi $x\in\xst$.
\ensep
Let us begin by noticing that both statements of the present theorem hold if $\flr_y=0$, that is, if $y\not\in X$. In this case statement~(b) is trivial, while statement~(a) holds because of Lemma~\ref{st:obsRosa}.(\fuu).
\ensep
Consider now the case $\flr_x=0$. In this case statement~(a) is empty, whereas statement~(b) reduces, via its contrapositive, to Lemma~\ref{st:obsRosa}.(\fuu) (with $x$ and $y$ interchanged with each other).

\halfsmallskip  
So, from now on, we can assume that $x$ and $y$ are both in $\xst$, or, on account of Theorem~\ref{st:zermelore}, that $\xst=\ist$.
In the following we will make use of the results of \secpar{2}, according to which the strengths $(\flr_x)$ are determined by the condition of maximizing the function (\ref{eq:funcioF}) under the restriction (\ref{eq:fratesaZ}), and that they satisfy the equations (\ref{eq:fratesZ}).

\halfsmallskip
Part~(a): It will be proved by seeing that a simultaneous occurrence of the inequalities $\flr_x > \flr_y$ and $\rho_x \le \rho_y$  would entail a contradiction with the fact that $\flr$ is the unique maximizer of $F(\flr)$. More specifically, we will see that one would have $F(\flrbis)\ge F(\flr)$ where $\flrbis$ is obtained from $\flr$ by interchanging the values of $\flr_x$ and $\flr_y$, that is
\begin{equation}
\label{eq:intercanvi}
\flrbis_z \,=\,
\begin{cases}
\flr_y, &\text{if }z=x,\\
\flr_x, &\text{if }z=y,\\
\flr_z, &\text{otherwise.}\\
\end{cases}
\end{equation}
In fact, $\flrbis$ differs from $\flr$ only in the components  associated with $x$ and $y$, so that
\begin{equation}
\begin{split}
\label{eq:quocientFs}
\frac{F(\flrbis)}{F(\flr)} \,=\,
&\left(\frac{\flrbisx}{\flr_x}\right)^{\hskip-3ptv_{xy}}
\prod_{z\ne x,y}\,
\left(\frac{\flrbisx/(\flrbisx+\flr_z)}{\flr_x/(\flr_x+\flr_z)}
\right)^{\hskip-3ptv_{xz}}
\left(\frac{\flr_x+\flr_z}{\flrbisx+\flr_z}\right)^{\hskip-3ptv_{zx}}\\
\times
&\left(\frac{\flrbisy}{\flr_y}\right)^{\hskip-3ptv_{yx}}
\prod_{z\ne x,y}\,
\left(\frac{\flrbisy/(\flrbisy+\flr_z)}{\flr_y/(\flr_y+\flr_z)}
\right)^{\hskip-3ptv_{yz}}
\left(\frac{\flr_y+\flr_z}{\flrbisy+\flr_z}\right)^{\hskip-3ptv_{zy}}.
\end{split}
\end{equation}
More particularly, in the case of (\ref{eq:intercanvi}) this expression becomes
\begin{equation}
\label{eq:canviphis-aF}
\frac{F(\flrbis)}{F(\flr)} \,=\,
\left(\frac{\flr_y}{\flr_x}\right)^{\hskip-3ptv_{xy}-v_{yx}}\,
\prod_{z\ne x,y}\,
\left(\frac{\flr_y/(\flr_y+\flr_z)}{\flr_x/(\flr_x+\flr_z)}
\right)^{\hskip-3ptv_{xz}-v_{yz}}
\left(\frac{\flr_y+\flr_z}{\flr_x+\flr_z}\right)^{\hskip-3ptv_{zy}-v_{zx}},
\end{equation}
where all of the bases are strictly less than~$1$, since
$\flr_x > \flr_y$, and all of the the exponents are non-positive, because of Lemma~\ref{st:obsRosa}.(\cuu,\,\ciu).
Therefore, the product is greater than or equal to $1$, as claimed.

\newcommand\flro{\omega}
\newcommand\incr{\epsilon}

\halfsmallskip
Part~(b): Since we are assuming $x,y\in\xst$, it is a matter of proving that $\rho_x > \rho_y \,\Rightarrow\, \flr_x > \flr_y$. On the other hand, by making use of the contra\-positive of~(a), the problem reduces to proving that $\flr_x = \flr_y \,\Rightarrow\, \rho_x = \rho_y$.

Similarly to above, this implication will be proved by seeing that a simultaneous occurrence of the equality $\flr_x=\flr_y=:\flro$ together with the inequal\-ity $\rho_x > \rho_y$ (by symmetry it suffices to consider this one) would entail a contradiction with the fact that $\flr$ is the unique maximizer of $F(\flr)$.\linebreak[3]
More specifically, here we will see that one would have $F(\flrbis)>F(\flr)$ where $\flrbis$ is obtained from $\flr$ by slightly increasing $\flr_x$ while decreasing $\flr_y$, that is
\begin{equation}
\label{eq:dissociacio}
\flrbisz \,=\,
\begin{cases}
\flro + \incr, &\text{if }z=x,\\
\flro - \incr, &\text{if }z=y,\\
\flr_z, &\text{otherwise.}\\
\end{cases}
\end{equation}
This claim will be proved by checking that
\begin{equation}
\label{eq:depspositiva}
\left.
\frac{\textup{d}\hphantom{\incr}}{\textup{d}\incr}\,
\log \frac{F(\flrbis)}{F(\flr)}
\,\right|_{\epsilon=0} \,>\, 0.
\end{equation}
In fact, (\ref{eq:quocientFs}) entails that
\begin{equation}
\begin{split}
\log\,\frac{F(\flrbis)}{F(\flr)} \,=\,\,\,
&C \,+\, v_{xy}\log\flrbisx + v_{yx}\log\flrbisy\\
+ &\sum_{z\ne x,y}\,\Big(
v_{xz}\log\frac{\flrbisx}{\flrbisx+\flr_z}
+ v_{yz}\log\frac{\flrbisy}{\flrbisy+\flr_z}\Big)\\
- &\sum_{z\ne x,y}\,\Big(
v_{zy}\log(\flrbisy+\flr_z)
+ v_{zx}\log(\flrbisx+\flr_z)\Big),
\end{split}
\end{equation}
where $C$ does not depend on $\incr$. Therefore, in view of (\ref{eq:dissociacio}) we get
\begin{equation}
\begin{split}
\left.
\frac{\textup{d}\hphantom{\incr}}{\textup{d}\incr}\,
\log \frac{F(\flrbis)}{F(\flr)}
\,\right|_{\epsilon=0}
\,=\,\,\,
(v_{xy}-v_{yx})\,\frac{1}{\flro}
\,\,+\, &\sum_{z\ne x,y}\,
(v_{xz}-v_{yz})\,\frac{\flr_z}{\flro(\flro+\flr_z)}\\
+\, &\sum_{z\ne x,y}\,
(v_{zy}-v_{zx})\,\frac{1}{\flro+\flr_z}.
\end{split}
\end{equation}
Now, according to parts (\cuu) and (\duu) of Lemma~\ref{st:obsRosa}, the assumption that $\rho_x>\rho_y$ implies the inequalities $v_{xy}>v_{yx}$, $v_{xz}\ge v_{yz}$ and $v_{zy}\ge v_{zx}$, which result indeed in~(\ref{eq:depspositiva}).
\end{proof}

\paragraph{3.4}
It is interesting to notice that a CLC~matrix keeps an important part of this structure when passing to the relative scores $q_{xy}=v_{xy}/t_{xy}$:

\begin{proposition}\hskip.5em
\label{st:relativescores}
Assume that $(v_{xy})$ is a CLC~matrix. If one has $t_{xy}>0$ for all $x,y$, then the relative scores $q_{xy}=v_{xy}/t_{xy}$ have the following properties, where $\xi$ is any admissible order for $(v_{xy})$:
\begin{alignat}{2}
\label{eq:qxyinequality}
&q_{xy} \,\ge\, q_{yx},
\qquad &&\hbox{whenever $x\rxi y$},
\\[2.5pt]
\label{eq:qinequalities}
&q_{xz} \,\ge\, q_{yz},\quad q_{zx} \,\le\, q_{zy},\qquad
&&\hbox{whenever $x\rxi y$ and $z\not\in\{x,y\}$}.
\end{alignat}
Besides, the top dominant irreducible component $\xst$ of $(v_{xy})$
is also top dominant irreducible for $(q_{xy})$,
with the special property that
\begin{equation}
\label{eq:irred-quo-positivitat}
q_{xy}\,>\,0,\qquad\mbox{whenever $x\in\xst$ and $y\ne x$.}
\end{equation}
If one has $t_{xy}=0$ for some $x,y$, then there exists $\yst\subseteq\ist$ such that 
\begin{align}
\label{eq:txx}
&t_{x\bar x}>0,\qquad \mbox{whenever\,  $x,\bar x\not\in\yst$,}\\[2.5pt]
\label{eq:zeroes}
&v_{yx}=0,\qquad\mbox{whenever\, $y\in\yst$ and\, $x\ne y$.}
\end{align}
\end{proposition}

\begin{proof}\hskip.5em
Consider first the case where $t_{xy}>0$ for all $x,y$. Clearly, (\ref{eq:vxyinequality}) immediately implies (\ref{eq:qxyinequality}). On the other hand, (\ref{eq:vinequalities}) implies (\ref{eq:qinequalities}) because of the following chains of implications:
\begin{align}
\label{eq:relativeinequalitiesA}
\frac{v_{xz}}{t_{xz}}\ge \frac{v_{yz}}{t_{yz}} \,\Longleftrightarrow\,
\frac{t_{xz}}{v_{xz}}\le \frac{t_{yz}}{v_{yz}} \,\Longleftrightarrow\,
1+  \frac{v_{zx}}{v_{xz}} \le 1+\frac{v_{zy}}{v_{yz}},
\\[2.5pt]
\label{eq:relativeinequalitiesB}
\frac{v_{zx}}{t_{zx}}\le \frac{v_{zy}}{t_{zy}} \,\Longleftrightarrow\,
\frac{t_{zx}}{v_{zx}}\ge \frac{t_{zy}}{v_{zy}} \,\Longleftrightarrow\,
1+  \frac{v_{xz}}{v_{zx}} \ge 1+\frac{v_{yz}}{v_{zy}}.
\end{align}
The statement about the top dominant irreducible component is also an immediate consequence of the positivity of the turnouts.

If $t_{xy}\!=\!0$ for some $x,y$, then (\ref{eq:tinequalities}) allows to derive that $t_{pp'}=0$ for some $p$. If $p_1$ is the first element with this property and we set $\yst=\{y\in\ist\mid p_1\rxieq y\}$, we immediately obtain (\ref{eq:txx}), and (\ref{eq:vinequalities}) together with (\ref{eq:vequaltomax}) are easily seen to lead to~(\ref{eq:zeroes}).
\end{proof}

\smallskip
As a consequence, we can see that the mean relative preference scores
\begin{equation}
\label{eq:srates}
\sigma_x \,=\, {\frac{\hbox{\small1}}{\hbox{\small{$N-1$}}}}\,\sum_{y\neq x} v_{xy}/t_{xy}
\end{equation}
rank the items in the same way as the original mean preference scores:

\smallskip
\begin{corollary}\hskip.5em
\label{st:rs}
Assume that $(v_{xy})$ is a CLC~matrix with positive turnouts.
In that case, one has $\sigma_x > \sigma_y$ \ifoi $\rho_x > \rho_y$.
\end{corollary}

\begin{proof}\hskip.5em
In order to prove the stated equivalence, it suffices to prove the two following implications:
\begin{gather}
\label{eq:rsa}
\rho_x \ge \rho_y \ensep\Longrightarrow\ensep \sigma_x \ge \sigma_y,
\\[2.5pt]
\label{eq:rsb}
\rho_x > \rho_y \ensep\Longrightarrow\ensep \sigma_x > \sigma_y.
\end{gather}
The implication~(\ref{eq:rsa}) is easily obtained by combining parts~(\cuu,\,\ciu) of Lemma~\ref{st:obsRosa} with 
the chain of implications (\ref{eq:relativeinequalitiesA}).
\ensep
In order to prove~(\ref{eq:rsb}) it suffices to notice that, according to part~(\duu) of Lemma~\ref{st:obsRosa}, $\rho_x > \rho_y$ implies $v_{xy}>v_{yx}$ and therefore $v_{xy}/t_{xy}>v_{yx}/t_{yx}$.
\end{proof}

\smallskip
\remark
If $t_{xy}=0$ for some $x,y$, the last statement of Proposition~\ref{st:relativescores} justifies considering any $y\in\yst$ categorically worse than any $x\not\in\yst$, and restricting the rating to the subset $\xst=\ist\setminus\yst$, which brings the problem to the case of positive turnouts.

\newcommand\zmap{Z}

\section{The CLC projection followed by Zermelo's method of strengths}

In this section we consider the rating method that is obtained by composing the CLC projection of \cite{crc,cri} and Zermelo's method of strengths.\linebreak 
That is, we consider the mapping $\Phi = \zmap P$, where $P$ denotes the CLC~projection mapping $(v_{xy})\mapsto(\psc_{xy})$ and $\zmap$ denotes the mapping defined by the method of strengths. 
As we will see, the properties of $\zmap$ obtained in the present article (Section~\ref{sec:zms}) combine with those of~$P$ (Section~\ref{sec:clc}) to ensure that the resulting rating method enjoys the properties that we claimed in the introduction.

We will use the following notations: $(v_{xy})$ denotes the original Llull matrix, $(\psc_{xy})$~denotes the projected one, and $(\flr_x)$ denotes the final strengths. We will refer to the latter as the \dfd{CLC-Zermelo fractions}.

\bigskip
The next result establishes the property of single-choice voting consistency.

\begin{theorem}\hskip.5em
\label{st:singlechoice}
When each ballot confines to choosing a single option, the CLC-Zermelo fractions coincide with the respective vote fractions.
\end{theorem}

\begin{proof}\hskip.5em
As it was mentioned in \secpar{1.1}, in the case of single-choice voting one has $v_{xy}=\scf_x$, and therefore $t_{xy}=\scf_x+\scf_y$, for every $y\neq x$. 
Such matrices have CLC~structure, so they are invariant by the CLC~projection:
$\psc_{xy}=v_{xy}=\scf_x$. By~plugging these values in~(\ref{eq:fratesZ}--\ref{eq:fratesaZ}), one easily sees that these equations are satisfied by taking $\flr_x=\scf_x$.
\end{proof}

\smallskip
Let us consider now the property of unanimous decomposition.
This property is concerned with \dfd{unanimously preferred sets}, \ie subsets $\xst$ of options with the property that each member of~$\xst$ is~unanimously preferred to any alternative from outside $\xst$.

\begin{theorem}\hskip.5em
\label{st:decomposition}
\textup{(a)}~The CLC-Zermelo fractions vanish outside of any unanimously preferred set.
\textup{(b)}~In the complete case, the options that get non-vanishing CLC-Zermelo fractions
form a minimal unanimously preferred set.
\end{theorem}

\begin{proof}\hskip.5em
The proof hinges on comparing the set under consideration with the top dominant irreducible component of the projected Llull matrix~$(\psc_{xy})$, whose existence is guaranteed by Proposition~\ref{st:existencia-X}. In the following, this top dominant irreducible component is denoted by $\xsth$.

\halfsmallskip
Part~(a).\ensep
Let $\xst$ be an unanimously preferred set. 
By \cite[Lem.~7.1]{cri}, the hypothesis that $v_{xy}=1$ for all $x\in\xst$ and $y\not\in\xst$ implies $\psc_{xy}=1$, and therefore $\psc_{yx}=0$, for all such pairs.\ensep
This entails that $\xsth\sbseteq\xst$, which leads to the claimed conclusion since Theorem~\ref{st:zermelore} ensures that $\flr_y=0$ for any $y\not\in\xsth$.

\halfsmallskip
Part~(b).\ensep
Let $\xst$ be the set of options with non-vanishing CLC-Zermelo fractions.
We claim that $\xst=\xsth$.
In~fact, otherwise Theorem~\ref{st:zermelore} would imply the existence of some $x\in\xst$ with $\flr_x=0$ or~some $y\in\yst=\ist\setminus\xst$ with $\flr_y>0$.

In particular, we have $\psc_{yx}=0$ for~all $x\in\xst$ and $y\in\yst$. Because of the completeness assumption, this implies that $\psc_{xy}=1$ and \hbox{---by \cite[Lem.~7.1]{cri}---} $v_{xy}=1$ for all those pairs. So $\xst$ is an unanimously preferred set.

Finally, let us see that $X$ is minimal for this property: If we had $\xstp\sbset\xst$ satisfying $v_{\xh\yh}=1$ for all $\xh\in\xstp$ and $\yh\in\ystp=\ist\setminus\xstp$, then \cite[Lem.~9.1]{crc} would give $\psc_{\xh\yh}=1$ and therefore $\psc_{\yh\xh}=0$ for all such pairs, so $\xst$ could not be the top dominant irreducible
component of the matrix $(\psc_{xy})$.
\end{proof}


Still in connection with the property of unanimous decomposition,
the following proposition includes a special case of incompleteness that has practical interest.

\begin{proposition}\hskip.5em
\label{st:decomposition-proposition}
Assume that the individual votes are complete, or alternatively, that each of them is a ranking (possibly truncated or with ties). If~$\xst$~is a minimal unanimously preferred set, then the CLC-Zermelo fractions of~$\xst$ are all of them positive and they coincide with those that one obtains when the individual votes are restricted to~$\xst$.
\end{proposition}

\begin{proof}\hskip.5em
%
Let us begin by noticing that the CLC~structure of the projected Llull matrix ensures that 
$\pto_{xy}=1$ for all $x,y\in\xst$. In fact, the inequalities (\ref{eq:tinequalities}) allow to derive it from the known fact ---obtained in the proof of part~(a) of the preceding theorem--- that $\pto_{xy}=1$ for all $x\in\xst$ and $y\not\in\xst$.

Now we claim that under the present hypotheses, \ie either completeness or ranking character of the individual votes, one has $\xsth=\xst$. In fact, a strict inclusion $\xsth\sbset\xst$ would mean that $\psc_{x\xh}=0$ for any $x\in\xst\setminus\xsth$ and $\xh\in \xsth$.
By the remark of the preceding paragraph, this implies that $\psc_{\xh x}=1$ for all such pairs. Since we also have $\psc_{xy}=1$ for $x\in\xst$ and $y\not\in\xst$, we can conclude that $\psc_{\xh\yh}=1$ for all $\xh\in\xsth$ and $\yh\not\in\xsth$. Now, according to \cite[Lem.~9.1]{crc} (for the complete case) and \cite[Lem.~7.1]{cri} (for the case of rankings, which are certainly transitive), this implies that $v_{\xh\yh}=1$ for all such pairs. This contradicts the supposed minimality of $\xst$.

So, $\xst$ itself is the top dominant irreducible component of the matrix $(\psc_{xy})$.
By making use of Theorem~\ref{st:zermelore}, it follows that $\flr_x>0$ for all $x\in\xst$ and that they are the strengths determined by the restriction of the projected Llull matrix $(\psc_{xy})$ to the set $\xst$. In order to 
complete the proof, 
we must show that this restriction of the projected Llull matrix coincides with the projection of the same  restriction applied to the original Llull matrix $(v_{xy})$, \ie $\psc_{x\bar x}=\pscbis_{x\bar x}$ for~any $x,\bar x\in\xst$, where we are using a tilde to denote the objects associated with the matrix obtained by first restricting and then projecting. 
In order to establish this equality, it suffices to obtain analogous equalities for the corresponding margins and turnouts. Besides, by taking into account the way that the CLC~projection is defined, it suffices to obtain these equalities for $\bar x=x'$, namely the option that immediately follows $x$ in an admissible order $\xi$ ($\xst$ is easily seen to be a segment of~$\xi$). For the margins, this equality is obtained in \cite[Lem.~9.2]{crc}, whose proof is valid without any need for completeness. For the turnouts, this equality is immediately true in the complete case. In the case of ranking votes, it suffices to observe that the $\xst$ restriction of the original Llull matrix is complete. This is true because of the following implications:\ensep (i)~$v_{xy}=1$ for some $y\in\ist$ implies that
$x$~is explicitly mentioned in all of the ranking votes;\ensep
and\,~(ii)~$x$~being explicitly mentioned in all of the ranking
votes implies that $t_{xy}=1$ for any $y\in\ist.$
\end{proof}

\medskip
\begin{theorem}\hskip.5em
\label{st:continuity}
The CLC-Zermelo fractions depend continuously on the original Llull matrix.
\end{theorem}

\begin{proof}\hskip.5em
This is a consequence of the continuity of the mappings $P$ and~$\zmap$. The former is guaranteed by \cite[Thm.\,6.1]{cri} and the latter by Theorems~\ref{st:zermelore} and \ref{st:zermeloirre}.
\end{proof}

\medskip
\begin{theorem}\hskip.5em
\label{st:majority}
The CLC-Zermelo fractions comply with the Condorcet-Smith principle:
If $\ist$ is partitioned in two sets $\xst$ and $\yst$ with the
property that $v_{xy} > 1/2$ for any $x\in\xst$ and $y\in\yst$,
\ensep
then for any such $x$ and $y$ one has either $\flr_x>\flr_y$ or $\flr_x=\flr_y=0$.
\end{theorem}

\begin{proof}\hskip.5em
Let us assume that the original Llull matrix is in the situation considered by the Condorcet-Smith principle. According to \cite[Thm.\,8.1]{cri},
the mean ranks $R_x$ of the projected Llull matrix $(\psc_{xy})$ satisfy the inequality $R_x < R_y$ for any $x\in\xst$ and $y\in\yst$. In terms of the mean preference scores~$\rho_x$, which are related to the mean ranks $\arank{x}= R_x$ by the linear decreasing transformation~(\ref{eq:arank}), we get therefore $\rho_x > \rho_y$ for any such $x$ and $y$.
So, it suffices to combine that result with Theorem~\ref{st:phis} of the preceding section.
\end{proof}

\medskip
Let us assume that all the individual preferences are reversed,
or equivalently, that the Llull matrix is replaced by its transpose. 
As a result of such a transformation,
one would expect the final ranking to be reversed.
This condition is known in the literature by the name of \dfc{inversion} \cite{che:1998,gonz}.

Zermelo's method by itself is easily seen to satisfy this condition, at least in the irreducible case.
More precisely, in this case the strengths for the transposed matrix are proportional to $1/\flr_x,$
where $\flr_x$ are the strengths for the original matrix.
In the reducible case, the positive strengths move from the top-dominant component to the bottom-dominated one (whenever the latter exists).

The following results establish the inversion property for the mean preference scores of the CLC-projected Llull matrix as well as for the CLC-Zermelo fractions (except for the ties between options with vanishing fractions).


\begin{theorem}\hskip.5em
\label{st:inversion-thm}
Assume that all of the binary preferences are reversed, \ie the scores $(v_{xy})$ are replaced by $(\vbis_{xy}),$ where $\vbis_{xy} = v_{yx}$. Let $\rho_x$ and $\rhobis_x$ be the mean preference scores of the respective CLC-projected Llull matrices. They behave in the following way: $\rho_x \,>\, \rho_y \,\Longrightarrow\,\rhobis_x \,<\, \rhobis_y$.
\end{theorem}

\begin{proof}\hskip.5em
Let us begin by noticing that the respective indirect scores satisfy
\begin{equation}
\label{eq:inversion-indirect}
\vbis^*_{xy} = v^*_{yx},
\end{equation}
which is clear from (\ref{eq:paths}).

Consider now the respective CLC-projected Llull matrices $(\psc_{xy})$ and $(\pscbis_{xy}).$
In \cite[\secpar{2.4}]{crc} and \cite[\secpar{2.1}, Step 2]{cri} it is seen that the admissible orders for $(\psc_{xy})$ are characterized as follows: $xy\in\xi$ if and only if $v^*_{xy} \ge v^*_{yx}$. Let us fix such an order~$\xi$ and let $\xibis$ be its reverse. Using (\ref{eq:inversion-indirect}), the preceding double implication can be rewritten as\, $yx\in\xibis$ if and only if $\vbis^*_{yx} \ge \vbis^*_{xy},$ which ensures $\xibis$ to be an admissible order for $(\pscbis_{xy}).$

We will now apply Lemma~\ref{st:obsRosa} to the CLC matrices $(\psc_{xy})$ and $(\pscbis_{xy}).$
Using part~(a) of that lemma ---as well as its contrapositive--- we get the following implication:
$\rho_x > \rho_y \Rightarrow xy\in\xi \Rightarrow yx\in\xibis \Rightarrow \rhobis_y \ge \rhobis_x$.

In order to complete the proof, it suffices to show that $\rhobis_y = \rhobis_x \Leftrightarrow \rho_x = \rho_y$. To this effect, we can restrict ourselves to the case where $y$ immediately follows $x$ in the order $\xi$. In fact, having $\rho_y=\rho_x$ and $x\rxi z\rxi y$ clearly implies $\rho_y=\rho_z=\rho_x$. So we can assume $y=x'$.
By part~(d) of Lemma~\ref{st:obsRosa}, our aim amounts at showing that $\pscbis_{xx'}=\pscbis_{x'x}\Leftrightarrow \psc_{xx'}=\psc_{x'x}$. Now, equations (20), (11) and (7) of \cite{cri} ensure that $\psc_{xx'}-\psc_{x'x} = \min\,\{\, v^*_{pq}-v^*_{qp} \;\vert\; p\rxieq x,\; x'\rxieq q\,\}.$ Using (\ref{eq:inversion-indirect}) as well as the fact that $\xibis$ is the reverse of $\xi,$ the right-hand side transforms into the analogous expression that corresponds to $\pscbis_{x'x}-\pscbis_{xx'}$. So we get $\pscbis_{x'x}-\pscbis_{xx'} = \psc_{xx'}-\psc_{x'x},$ which entails the claimed double implication.
\end{proof}

\remark
In the complete case these arguments together with the formulas for $\psc_{xy}$ given in \cite{crc} show that the CLC projection commutes then with transposition. This commutability is not true in the general incomplete case.

\begin{corollary}\hskip.5em
\label{st:inversion-cor}
In the conditions of the preceding theorem the CLC-Zermelo fractions behave in the following way:
$\flr_x \,>\, \flr_y \,\Longrightarrow\, \hbox{either\, }\flrbis_x < \flrbis_y \hbox{ \,or\, }\flrbis_x = \flrbis_y = 0$.
\end{corollary}

\begin{proof}\hskip.5em
Again, it suffices to combine the preceding result with Theorem~\ref{st:phis} of the preceding section.
\end{proof}

\newcommand\binrel{\eta}

\medskip
The next theorem is concerned with clone consistency.
In this connection we make use of the notion of autonomous sets.
\ensep
A subset $\cst\sbseteq\ist$ is said to be {\df autonomous} for a binary relation~$\binrel$ when
each element from outside~$\cst$ relates to all elements of~$\cst$ in the same way;
more precisely, when, for any~$x\not\in\cst$, having
$ax\in\binrel$ for some $a\in\cst$ implies $bx\in\binrel$ for any
$b\in\cst$, and similarly, having $xa\in\binrel$ for some $a\in\cst$
implies $xb\in\binrel$ for any $b\in\cst$.
\ensep
More generally, a subset $\cst\sbseteq\ist$ will be said to be autonomous for a valued relation $(v_{xy})$ when
the equalities $v_{ax} = v_{bx}$ and $v_{xa} = v_{xb}$ hold whenever
$a,b\in\cst$ and $x\not\in\cst$.
\ensep
For more details about the notion of autonomous set and the property of clone consistency we refer the reader to \cite[\secpar{11}]{crc}.
\ensep
Autonomous sets are also considered in \cite{che:1994}, where they are called macrovertices.

\medskip
\begin{theorem}\hskip.5em
\label{st:clones}
The CLC-Zermelo fractions have the following property of clone consistency: Assume that $\cst\sbset\ist$ is an autonomous set for each of the individual votes. Assume also that either $\cst\sbseteq\xst$ or $\cst\spseteq\ist\setminus\xst$,
where $\xst=\{\,x\in\ist\mid\flr_x>0\}$.
\ensep
Under these hypotheses one has the following facts: \textup{(a)}~$\cst$ is autonomous for the ranking determined by the CLC-Zermelo fractions;\, and\, \textup{(b)}~contracting $\cst$ to a single option in all of the individual votes has no other effect in that ranking than getting the same contraction.
\end{theorem}

\begin{proof}\hskip.5em
Once more, it suffices again to combine \cite[Thm.\,8.2]{cri} with Theorem~\ref{st:phis} of the preceding section.
\end{proof}

\medskip
Finally, the following 
result considers the effect of raising a particular
option $a$ to a more preferred status in the individual ballots
\textit{without any change in the preferences about the other options}.

\smallskip
\begin{theorem}\hskip.5em
\label{st:monotonicity}
Assume that the scores~$v_{xy}$ are modified into new values $\vbis_{xy}$ such that
\begin{equation}
\label{eq:mona}
\vbis_{ay} \ge v_{ay},\quad \vbis_{xa} \le v_{xa},\quad \vbis_{xy} = v_{xy},\qquad \forall x,y\neq a.
\end{equation}
In these circumstances the CLC-Zermelo fractions behave in the following way:
$\flr_a > \flr_y \,\Longrightarrow\, \flrbis_a \ge \flrbis_y$.
\end{theorem}

\begin{proof}\hskip.5em
According to \cite[Thm.\,8.4]{cri},
the mean ranks $R_x$ of the projected Llull matrix $(\psc_{xy})$ behave in the following way:
$R_a < R_y \,\Longrightarrow\, \widetilde R_a \le \widetilde R_y$.
In~terms of the mean preference scores $\rho_x$,
which are related to the mean ranks $\arank{x}= R_x$ by the linear decreasing transformation~(\ref{eq:arank}),
we have therefore $\rho_a > \rho_y \,\Longrightarrow\, \rhobis_a \ge \rhobis_y$.
So, it suffices to combine this with Theorem~\ref{st:phis} of the present article.
\end{proof}

\section{Concluding remarks}

\newcommand\ev{\tau}
\newcommand\fb{\psi}
\newcommand\is{\chi}

\paragraph{5.1}
In this subsection we look at the possibility of achieving the same properties by means of other methods.

A classical idea to be considered in this connection is rating the options by means of a \emph{non-negative right eigenvector} of the paired-comparison matrix. This approach arises naturally as a refinement of the mean preference scores. In~fact, the mean preference score of $x$ combines the preference scores $v_{xy}\ (y\neq x)$ with equal weights. However, one can argue that a given preference score over a highly rated option $y$ should convey more value to~$x$ than the same preference score over a lowly rated option $y$. This leads to looking for a system of ratings $(\ev_x)$  that satisfy a relationship of the form $\sum_{y\neq x} v_{xy}\,\ev_y = \lambda\,\ev_x$ for some $\lambda > 0$. In conformity with the idea of mixing proportions, one requires also $\ev_x\ge0$ and $\sum_{x} \ev_x = 1$. In other words, $(\ev_x)$ should be a non-negative right eigenvector of the matrix that is obtained from $(v_{xy})$ by filling the diagonal with zeroes, and the corresponding eigenvalue should be positive. 
One can look for such an eigenvector by solving the preceding equations in a direct way.
Alternatively, one can often approach it by an iterative procedure of the form $\ev^{(n+1)}_x = \sum_{y\neq x} v_{xy}\,\ev^{(n)}_y$
starting from a positive vector $\ev^{(0)}$. Usually one takes $\ev^{(0)}_x = 1/N$ for all $x$ (recall that $N$ is the number of options) in which case the $\ev^{(1)}_x$ are proportional to the mean preference scores.

This idea was put forward in 1895 by Edmund Landau in his first published mathematical paper \cite{la:1895}. Landau was motivated by chess tournaments, where some rating methods had been introduced that amounted to using the rating $\ev^{(2)}$.
He returned to the subject in 1914 \cite{la:1914}, after Oskar Perron and Georg Frobenius had proven
their celebrated theorem that guarantees the existence and uniqueness of such a non-negative eigenvector in the case of an irreducible non-negative matrix.
Forty years later, the same idea was proposed by
T.\,H.~Wei and Maurice G.~Kendall \cite{wei:1952,kendall:1955}.%
\footnote{We have not been able to check reference \cite{wei:1952}.}

In his second paper on the subject, Landau considered the condition of unanimous decomposition,
more specifically its part~(a) as stated in page~\pageref{txt:ud}, 
as~a~natural constraint for selecting among the several non-negative eigenvectors that can exist in the event of an unanimous decomposition \cite[p.\,201]{la:1914}. 
However, one can see that this constraint conflicts with the condition of continuity that we would like also to be satisfied.
Let us take, for instance, the following matrix:
\renewcommand\labelcell[1]{\cellcolor[gray]{0.8}\makebox[1.6em][c]{\skl{#1}}}
\begin{equation}
\label{eq:redeps}
\vaa_\epsilon \,=\,  
\begin{small}
\begin{tabular}{|c|c|c|}
\hlinestrut
\labelcell{a}&$1\cd-\epsilon$&$1\cd-\epsilon$\\
\hlinestrut
$\epsilon$&\labelcell{b}&\onehalf\\
\hlinestrut
$\epsilon$&\onehalf&\labelcell{c}\\
\hline
\end{tabular}
\end{small}
\end{equation}
For $\epsilon > 0$ (and less than $1$) its unique non-negative eigenvector (unique up to multiplication by a positive number) is $\big(\,8(1-\epsilon)(1+\sqrt{1\cd+32\epsilon\cd-32\epsilon^2}\,)^{-1},1,1\big)^{\textsf{T}}\!,$ whose limit as $\epsilon\downarrow0$ is $(4,1,1)^{\textsf{T}}$. However, this is not a multiple of~$(1,0,0)^{\textsf{T}},$ the rating that the condition of unanimous decomposition  requires for $\epsilon=0$.
(Notice that both $(4,1,1)^{\textsf{T}}$ and $(1,0,0)^{\textsf{T}}$ are non-negative eigenvectors of $\vaa_0$ and that their corresponding eigenvalues are respectively $\onehalf$ and $0$.)




\medskip
This and other problems seem to disappear for a method that can be viewed as a derivation of the preceding one, namely the so-called \emph{fair bets} method, proposed more or less independently by Henry E.~Daniels in 1969~\cite{da} and by John W.~Moon and Norman J.~Pullman in 1970~\cite{mp}.

The fair~bets, that we will denote by $\fb_x$, have the following meaning:
Let us interpret the paired-comparison scores $v_{xy}$ as numbers of victories of $x$ over $y$.
We will assume that every time that a player $x$ beats another one~$y$, the latter pays to the former the amount $\fb_y$. The fair bets are the values that result in no player winning nor losing any money.
In other words, for every~$x$ one should have the equality $\sum_{y\ne x} v_{xy} \fb_y = \sum_{y\ne x} v_{yx} \fb_x$. As before, together with these equations one requires also $\fb_x\ge0$ and $\sum_x\fb_x = 1.$



The fair bets are easily seen to have good behaviour in connection with the conditions of single-choice voting consistency and unanimous decomposition. A preliminary exploration suggests that they also depend continuously on the preference scores even in the neighbourhood of a reducible matrix. For instance, in the case of (\ref{eq:redeps}) they are proportional to $(1-\epsilon,\epsilon,\epsilon)^{\textsf{T}}.$


On the other hand, they do not satisfy the Condorcet principle (in common with the mean preference scores and Zermelo's strengths). For instance, in the case of (\ref{eq:ex1}--\ref{eq:llull1}) 
one gets the following values: \cand{a}: 0.323, \cand{b}: 0.378, \cand{c}: 0.174 \cand{d}: 0.124,
where \cand{b} gets the largest fraction in spite of the fact that \cand{a} has a majority of first placings.

However, the experience of this article on Zermelo's method suggests that
combining the CLC~projection with the fair-bets method
could also give a method with the desired properties,
including the Condorcet-Smith principle.
By the way, in the case of (\ref{eq:ex1}--\ref{eq:llull1}), this combined procedure gives the following results: \cand{a}:~0.325, \cand{b}:~0.286, \cand{c}:~0.214 \cand{d}:~0.175.
So we pose the two following questions:


\begin{open}\hskip.5em
Are the fair bets continuous functions of the preference scores even in the neighbourhood of a reducible matrix?
\end{open}

\begin{open}\hskip.5em
Is the Condorcet-Smith principle satisfied when the fair bets are preceded by the CLC~projection?
\end{open}

In contrast to Zermelo's method, the fair bets are known to violate the condition of inversion \cite[Example~4.4]{gonz}. However, numerical experiments suggest that
the following question may still have a positive answer:


\begin{open}\hskip.5em
Do the fair bets have the property of inversion when they are preceded by the CLC~projection?
\end{open}


%



\paragraph{5.2}
In Theorem~\ref{st:monotonicity} we considered the effect of raising a particular
option $a$ to a more preferred status in the individual ballots
\textit{without any change in the preferences about the other options}.
Besides the property that was obtained in that theorem,
in this situation it would be quite desirable to have an increase in the fraction associated with $a:$
$\flrbis_a \ge \flr_a.$
This condition of \dfc{quantitative monotonicity} was considered by Landau in \cite{la:1914},
where it is shown 
that this condition is violated by the right non-negative eigenvector even in the irreducible case.

Zermelo's method is ensured to have this property \cite[p.\,444]{ze}.
However, this is not true for Zermelo's method preceded by the CLC projection,
since the latter does not have good properties in this connection
(which motivated the open question~1 of \cite{crc}).

\end{document}